\documentclass[12pt]{amsart}
\usepackage[utf8]{inputenc}
\usepackage[english]{babel}
\usepackage{amsmath, amsthm, amssymb, amsfonts, enumerate, mathtools}
\usepackage[colorlinks=true,linkcolor=blue,urlcolor=blue]{hyperref}
\usepackage[square,numbers]{natbib}	
\usepackage{dsfont}
\usepackage{color}
\usepackage{geometry}
\usepackage{todonotes}
\usepackage{epstopdf}
\usepackage{bbm}
\usepackage{bm}
\usepackage{textcomp}
\usepackage{float}
\usepackage{tikz}
\usepackage{epsfig}
\usepackage{a4}
\usepackage{enumerate}

\newcommand{\stkout}[1]{\ifmmode\text{\sout{\ensuremath{#1}}}\else\sout{#1}\fi}

\definecolor{darkviolet}{rgb}{0.58, 0.0, 0.83}

\def\<{\left\langle }
\def\>{\right\rangle }

\usepackage{float}
\usepackage{soul}
\newtheorem{theorem}{Theorem}[section]
\newtheorem{proposition}[theorem]{Proposition}
\newtheorem{lemma}[theorem]{Lemma}
\newtheorem{assumption}[theorem]{Assumption}
\newtheorem{corollary}[theorem]{Corollary}
\newtheorem{remark}[theorem]{Remark}
\newtheorem{definition}[theorem]{Definition}
\newtheorem{Lemma}[theorem]{Lemma}

\def\R{\mathbb R}

\def\d{\mathrm{d}}

\def\cala{{\mathcal A}}

\def\calc{{\mathcal C}}

\def\calu{{\mathcal U}}

\newcommand{\de}{{\ensuremath{\mathrm e}} }
\newcommand{\dL}{{\ensuremath{\mathrm L}} }

\def \cC{{\mathcal C}}

\def \R{\mathbb{R}}

\def\d{\mathrm{d}}

\DeclareMathOperator*{\argmax}{argmax}

\definecolor{red}{rgb}{1.0,0.0,0.0}

\definecolor{blu}{rgb}{0.0,0.0,1.0}
\def\blu#1{{\textcolor{blu}{#1}}}
\definecolor{gre}{rgb}{0.03,0.50,0.03}

\newcommand{\disableblucomments}{\renewcommand{\blu}{}}

\makeatletter
\def\@setthanks{\vspace{-\baselineskip}\def\thanks##1{\@par##1\@addpunct.}\thankses}
\makeatother

\newcommand{\mail}[1]{\href{mailto:#1}{\normalfont\texttt{#1}}}

\title[Constrained control problems in Banach lattices]{State constrained control problems\\ in Banach lattices and applications}

\author[A.~Calvia ]{Alessandro Calvia\textsuperscript{\MakeLowercase{a},1}}
\thanks{\noindent \textsuperscript{a} Dipartimento di Economia e Finanza, LUISS \emph{Guido Carli}, Roma.}

\author[S.~Federico ]{Salvatore Federico\textsuperscript{\MakeLowercase{b},2}}
\thanks{\noindent \textsuperscript{b} Dipartimento di Economia, Università degli Studi di Genova.}

\author[F.~Gozzi ]{Fausto Gozzi\textsuperscript{\MakeLowercase{a},3}}
\thanks{\noindent \textsuperscript{1} E-mail: \mail{acalvia@luiss.it}.
\\
\noindent \textsuperscript{2} E-mail: \mail{salvatore.federico@unige.it}
\\
\noindent \textsuperscript{3} E-mail: \mail{fgozzi@luiss.it}}

\numberwithin{equation}{section}

\begin{document}

\frenchspacing

\disableblucomments

\begin{abstract}
This paper aims to study a family of deterministic optimal control problems in infinite dimensional spaces. The peculiar feature of such problems is the presence of a positivity state constraint, which often arises in economic applications.
To deal with such constraints, we set up the problem in a Banach space with a Riesz space structure (i.e., a Banach lattice) not necessarily reflexive: {a typical example is the space of continuous functions on a compact set.}
In this setting, which seems to be new in this context, we are able to find explicit solutions to the Hamilton-Jacobi-Bellman (HJB) equation associated to a suitable auxiliary problem \blu{and to write the corresponding optimal feedback control. Thanks to a type of infinite dimensional Perron-Frobenius Theorem, we
use these results} to get information about the optimal paths of the original problem.
This was not possible in the infinite dimensional setting \blu{used in earlier works on this subject}, where the state space was an $\dL^2$ space.
\end{abstract}

\maketitle


\bigskip

\noindent \emph{AMS 2020:} 46B42, 49K27, 49L20, 93C20, 93C25

\bigskip

\section{Introduction} \label{sec:introduction}
A typical feature of optimal control problems arising in economic applications --- as well as in other fields --- is the presence of positivity state and/or control constraints. This feature is very common and makes the problem nontrivial, {even in the case where} the state space has finite dimension. Indeed, 
on the one hand, \blu{if one approaches the problem by the Maximum Principle, then jumps of the co-state variables, associated to optimality conditions, may arise} (see, e.g., \citet{hartl1995:mp} in finite dimension and \citet[Ch. 11]{Fattorini99} in infinite dimension);
on the other hand, \blu{if one employs Dynamic Programming techniques, it is hard} to get well-posedness and regularity results for the associated HJB equation (see, e.g., \citet{capuzzolions1990:hjb,soner:optcontrol1} in finite dimension; \citet{CannarsaDiBlasio95, CannarsaGozziSoner91, Faggian08, KocanSoravia98} in infinite dimensions; 
see also, in the stochastic case, \citet{Katsoulakis94,
% \red{SF: CONTROLLARE SE QUESTI RIFERIMENTI SONO OK QUI: C'E' VINCOLO DI STATO? [AC: nel lavoro su SICON c'è, anche se il contesto è un po' più semplice. Ad ogni modo studio l'HJB che risulta adattando il metodo di Soner nei suoi lavori con vincolo di stato. Nel lavoro su ESAIM scrivo solo alla fine in un remark l'HJB ma non la studio (ci sarebbe comunque il vincolo di stato). Se reputate opportuno si può togliere il secondo lavoro, anche se preferirei lasciarlo per ovvi motivi. Nel lavoro di Mauro a una prima occhiata mi sembra che non ci sia un vincolo, ma credo lo sappia meglio Fausto. Stessa cosa nel capitolo citato del libro.]} 
calvia:MCcontrol, calvia:filtcontrol}, and the book by \citet[Ch. 3]{fabbri:soc}).

\blu{In the last decades, various papers in the economic literature considered optimal control problems where the state variable is  infinite dimensional. This typically happens when one has to take into account heterogeneity: for instance, the evolution of a key economic variable like capital may depend not only on time, but also on its spatial position or on its vintage  (see, e.g., \citet{BFG, BDFG,  boucekkine:spatialAK, BFFGjoeg,  BFFGpafa, feichtinger2006:vintage}, \citet{FabbriGozzi08,BoucekkineetalJET2005}). The setting where such optimal control problems were studied in this stream of literature is the separable Hilbert space $\dL^2$. Also in this case, positivity state constraints represent an essential feature of the problem. Handling them in this infinite dimensional setting turns out to be even more problematic than in finite dimension, as the positive cone of $\dL^2$ has empty interior and it is not obvious how to interpret the derivative in the associated HJB equation. Consequently,  it is not obvious to give sense to the optimal feedback map, which, as well known, depends on the derivative of the solution to the HJB equation.}

A possible way to deal with such difficulty (cf., e.g., \citet{BFG,BDFG}, \citet{boucekkine:spatialAK, BFFGjoeg}, \citet{FabbriGozzi08,BoucekkineetalJET2005}) is to study an auxiliary problem \blu{under a relaxed} state constraint, by allowing the state variable to evolve in a suitable half-space where the solution of the HJB equation can be found explicitly.
\blu{If it is possible to find the optimal paths for the auxiliary problem, one can try to show that, at least for some initial data (hopefully, interesting for applications), these  paths satisfy the positivity state constraint, and hence are optimal for the initial problem, too. This important property, i.e., the admissibility of the auxiliary optimal paths for the original problem, has been rigorously established in \citet{boucekkine:spatialAK, BFFGjoeg} only for the \emph{steady states} of the problem, that constitute a one dimensional set of initial conditions.} To the best of our knowledge, only in \cite{BDFG} this point is successfully addressed for a larger nontrivial set of initial conditions, but at the price of a technical, involved, and tailor-made argument.
\blu{Thus, the main motivation of our paper is to rigorously prove that admissibility of the auxiliary optimal paths for the original problem holds for a sufficiently rich set of initial conditions, in the context of economic growth problems in time-space like the ones introduced \citet{boucekkine:spatialAK, BFFGjoeg}: these results  are collected in Section~\ref{sec:app} (cf. Theorem~\ref{th:econconv2} and Corollary~\ref{th:econconvconst}).}

%On the other hand, it is clearly important to have results guaranteeing that, at least when the initial conditions are in a neighborhood of the corresponding steady states,
%the auxiliary optimal paths \blu{satisfy the positivity state constraint (and, thus, they are optimal for the original problem), and that they converge to that steady state. The main motivation of our paper is to prove such results for economic  growth problems in time-space like the ones introduced \citet{boucekkine:spatialAK, BFFGjoeg}: these results  are collected in Section \ref{sec:app} (cf. Theorem~\ref{th:econconv2} and Corollary \ref{th:econconvconst}).}

\blu{To achieve this goal, we consider the problems in a state space different from $\dL^2$; since heterogeneity is usually modeled by a compact metric space $D$, the natural candidate to use is the space of continuous functions on $D$, endowed with the \emph{sup-norm} and the standard pointwise partial order. This norm allows to deal properly with pointwise constraints, such as the positivity state constraint that characterizes our optimization problems}.
\blu{However, we realized that, for the class of control problems with the features described above, many results holding in the setting of continuous functions still hold in the more general setting of Banach lattices. Since we think that this more general setting can be useful for future applications, we build our theory  in this abstract setup. More specifically, we consider:}
\begin{itemize}
  \item[--] as state space $X$, a general separable Banach lattice of real valued functions defined on a measure space $(D,\mu)$; the strictly positive orthant $X_{++}$ is naturally defined and provides the required positivity state constraint;\vspace{2pt}
  \item[--] as state equation, the linear equation  $x'(t)=Lx(t)-Nc(t)$, where  $x$ and $c$ are, respectively, the state and the control variable and  $L$ and $N$ are suitable operators;\vspace{2pt}
  \item[--] as objective functional, a type of discounted utility over infinite horizon, i.e., a functional of the form
      $\int_{0}^{\infty} \de^{-\rho t}\mathcal{U}(c(t)) \, \d t$,
      where $\mathcal{U}$ is concave.
\end{itemize}
\blu{In this general setup, we define an appropriate auxiliary problem where the state constraint is a suitable half-space of $X$, that strictly contains the positive cone. For this problem we prove a verification theorem (Theorem~\ref{verification}). Moreover, in the case 
when $X$ is either an $\dL^{p}$ space or the space of continuous function and $\mathcal{U}$ is homogeneous, we show that the HJB equation can be explicitly solved (Proposition~\ref{prop:exp}), and, consequently, that the optimal feedback control can be found (Theorem~\ref{teo:main}).
\footnote{\blu{It is worth noticing that HJB equations in Banach spaces have been scarcely studied in the literature (see, e.g., \citet{addona2020:BE,fuhrman2010:SDDE,masiero2008:soc,masiero2016:hjb} for mild solutions to second order semilinear HJB equations; \citet{crandall85:HJB1, crandall86:HJB2, crandall86:HJB3, soner88:HJB}, for viscosity solutions to first order  HJB equations).
}}}

\blu{Once the auxiliary problem in the half-space is explicitly solved, we are in the position to achieve our main goal, i.e., to prove a stability result (Theorem~\ref{th:stability}) that provides the basis to show admissibility of the auxiliary optimal paths for the original problem, at least when the initial conditions are in a neighborhood of the \blu{corresponding steady state} (Theorem~\ref{th:econconv2} and Corollary~\ref{th:econconvconst}).}
\blu{To get these results, we need an important assumption that concerns the spectrum of the operator $L$ appearing in the state equation above (cf. point~\textit{(\ref{ass:eigenvect})} of Assumption~\ref{ass:explicit}). This assumption can be verified thanks to  Perron-Frobenius-type results in Banach spaces, which are provided in Section~\ref{sec:perronfrobenius}.}

The plan of the paper is the following.
\begin{itemize}
\item[--] Section~\ref{sec:model} is devoted to present the general setup mentioned above;
\item[--] Section~\ref{sec:hjb} contains the main abstract results of the paper on the auxiliary problem: the verification theorem (Subsection~\ref{sec:ver}), the explicit solution to the auxiliary HJB equation (Subsection~\ref{sec:hjbexpl}), the stability result (Subsection~\ref{sec:ss});
\item[--] Section~\ref{sec:perronfrobenius} introduces the Perron-Frobenius-type results \blu{guaranteeing the aforementioned key assumption  needed in Section~\ref{sec:hjb};}
\item[--] Section~\ref{sec:app} shows how the general theory exposed in the first part of the paper can be applied to a family of economic growth problems, including the one studied in \citet{boucekkine:spatialAK, BFFGjoeg}.
\item[--] \blu{Section~\ref{sec:conclusion} is devoted to discuss potential extensions of our approach to optimal control problems with positivity state constraints and nonlinear state equations.}

\end{itemize}

\section{The optimal control problem} \label{sec:model}
\blu{Set $\R_+ \coloneqq [0,+\infty)$, let $(D,\mu)$ be a measure space {with a countably generated $\sigma$-algebra}\footnote{In most applied examples $D$ has a topological structure. However, at this abstract stage, this is not needed.}, and let $(X, |\cdot|_X, \leq_X)$ be a separable Banach lattice of real-valued functions defined on $D$. The symbols $|\cdot|_X$ and $\leq_X$ denote, respectively, the norm and the ordering on $X$. }

\blu{Denote by $X^\star$ the topological dual of $X$ and let $\langle\cdot,\cdot\rangle$ be the dual pairing of $X, X^\star$. We have that $X^\star$ is an order-complete Banach lattice with its usual norm, denoted by $|\cdot|_{X^\star}$, {and with the natural ordering $\leq_{X^\star}$ defined as follows (see, e.g., \citep[pag. 239]{arendt1986:possgr}): given ${\varphi^\star}, {\psi^\star} \in X^\star$,
\begin{equation*} 
{\varphi^\star} \leq_{X^\star} {\psi^\star} \quad \Longleftrightarrow \quad \langle f, {\varphi^\star} \rangle \leq \langle f, {\psi^\star} \rangle, \quad \forall f \geq_X 0.
\end{equation*}
}}
\blu{We consider the positive orthants (or positive cones) of $X,X^\star$, i.e., the sets
\begin{equation*}
X_+ \coloneqq \{f\in X \colon \  f\geq_X 0\}, \ \ \ \ \ \  
X^\star_+ \coloneqq \{{\varphi^\star}\in X^\star \colon \ \varphi^\star \geq_{X^{\star}} 0\},
\end{equation*}
and the strictly positive orthants (or strictly positive cones) of $X,X^\star$, i.e., the sets (see e.g. \citep[pag. 119]{arendt1986:possgr})
\begin{align*}
X_{++} &\coloneqq \{f\in X \colon \  \langle f,{\varphi^\star}\rangle>0, \ \  \ \forall\, 0\neq {\varphi^\star}\in {X^\star_+}\}, 
\\
X^\star_{++} &\coloneqq \{{\varphi^\star}\in X^\star \colon \  \langle f, {\varphi^\star} \rangle > 0, \ \ \ \forall\,{0 \neq f \in X_+}\}.
\end{align*}
We write $f >_X 0$ {if $f\in X_{++}$ and ${\varphi^\star} >_{X^\star}0$  if ${\varphi^\star}\in X^\star_{++}$.}
}
\blu{\begin{remark}
The following spaces of functions, typical in applications, fall within the abstract setting above:
\begin{enumerate}[(i)]
\item $X = \dL^p(D, \mu)$, $p \in [1,+\infty)$, where $\mu$ is $\sigma$-finite. In this case $X$ is separable, $X^\star = \dL^q(D, \mu)$, where $q$ satisfies $p^{-1} + q^{-1} = 1$, and $\mathrm{int}\,X_+ = \emptyset$.
\item $X = \calc(D)$, the space of real-valued continuous functions on a compact metric space $D$, equipped with the sup-norm; {$\mu$ can be any Borel measure on $D$.} In this case, $X$ is a Banach lattice with order unit, $X^\star$ is the space of regular Borel measures on $D$, and $\mathrm{int}\,X_+ =X_{++} \neq \emptyset$.
\item $X = \calc_0(D)$, the space of real-valued continuous functions vanishing at infinity on a locally compact metric space $D$, equipped with the sup-norm; {$\mu$ can be any Borel measure on $D$.} In this case, $X^\star$ is the space of regular Borel measures on $D$, and  $\mathrm{int}\,X_+ = \emptyset$.
\item Given a locally compact metric space $D$ and a continuous function $w \colon D \to (0,+\infty)$, the space $X = \calc_w(D)$ of real-valued continuous functions $f$ on $D$ such that $fw$ is bounded, equipped with the norm
$$|f|_{X}:=\sup_{x\in D} \left|{f(x) w(x)}\right|.$$
Also here $\mu$ can be any Borel measure on $D$. A typical case is  $D=\R^{n}$ and $w(x)=\frac{1}{1+|x|^{k}}$, which allows to deal with continuous data with polynomial growth at infinity.
%\red{[AC: Serve davvero aggiungere questo esempio? Mi sembra si perda più tempo a definirlo rispetto al beneficio che se ne ricava (ossia rispondere al referee \#1), a meno che non si spieghi perché è importante questo spazio. Magari si può solo spiegarlo nella lettera al referee.]}
\end{enumerate}
\end{remark}}

\blu{We introduce the optimal control problem in the space $X$ that we aim to study.
Let $L \colon D(L)\subseteq X\to X$ be a (possibly) unbounded linear operator and let $N \colon X\to X$ be a bounded linear operator.
Given $x_0\in X$ and a control function $c\in \dL^1_{loc}(\R_+;X)$, we consider the following abstract \emph{state equation} in $X$:
\begin{equation}\label{state}
\left\{
\begin{aligned}
&x'(t)=Lx(t)-Nc(t), \quad t \geq 0, \\
&x(0)=x_0.
\end{aligned}
\right.
\end{equation}
To stress the dependence of the solution to \eqref{state} on $x_0 \in X$ and $c\in \dL^1_{loc}(\R_+;X)$, we will denote it by $x^{x_0, c}$.
The following assumption will be in force throughout the paper.
\begin{assumption}\label{ass:L}
\begin{enumerate}[(i)]
\item[]
\item\label{hyp:L} The linear operator $L \colon D(L)\subseteq X\to X$ is closed, densely defined, and generates a \mbox{$C_0$-semigroup} $\{\de^{{tL}}\}_{t \geq 0}$ in $X$. Moreover, the semigroup $\{\de^{{tL}}\}_{t \geq 0}$ {preserves strict positivity}, i.e., $\de^{{tL}}(X_{++})\subseteq X_{{++}}$, for all $t\geq 0$.
\item\label{hyp:N} The linear operator $N \colon  X\to X$ is bounded and positive, i.e., $N(X_+) \subseteq X_+$.
\end{enumerate}
\end{assumption}
\begin{remark}
{It is worth noticing that, by continuity, point~\textit{(\ref{hyp:L})} of Assumption~\ref{ass:L} implies that $\{\de^{{tL}}\}_{t \geq 0}$ is also a positive semigroup; that is, it satisfies $\de^{{tL}}(X_{+})\subseteq X_{{+}}$, for all $t\geq 0$.
Sufficient conditions guaranteeing positivity of semigroups can be found, e.g., in \citep[B-II, Th. 1.6, Th. 1.13, C-II, Th. 1.2, Th. 1.8]{arendt1986:possgr}) or \citep[Th. 7.29, Prop. 7.46]{clement1987:sgr}.}
\end{remark}
According to \cite[p. 129]{BDDM}, for each $x_0 \in X$ and $c\in \dL^1_{loc}(\R_+;X)$, {we call \emph{mild solution} to \eqref{state} the function}
\begin{equation}\label{mild}
x^{x_0, c}(t)\coloneqq\de^{tL}x_0-\int_0^t \de^{(t-s)L} Nc(s) \, \d s, \quad {t \geq 0.}
\end{equation}
By \cite[p. 131]{BDDM}, {the mild solution defined in} \eqref{mild} is also a \emph{weak solution} to \eqref{state}, i.e., {it satisfies, for all $\varphi^\star\in D(L^\star)$,}
\begin{equation}\label{weak}
\langle x^{x_0, c}(t),{\varphi^\star}\rangle = \langle x_0,{\varphi^\star}\rangle +\int_0^t \langle x^{x_0, c}(s), L^\star{\varphi^\star}\rangle \, \d s-\int_0^t\langle Nc(s),{\varphi^\star}\rangle \, \d s, \quad {t \geq 0,}
\end{equation}
where $L^\star \colon D(L^\star)\subseteq X^\star\to X^\star$ denotes the adjoint of $L$.}

\blu{We are interested in analyzing an optimal control problem where the state variable, satisfying \eqref{state}, remains in the strictly positive cone of $X$. This state constraint is expressed introducing the following class of admissible controls, depending on $x_0 \in X$:
\begin{equation}\label{eq:A++def}
\mathcal{A}_{{++}}(x_0)\coloneqq\{c\in  \dL^1_{loc}(\R_+;X_+) \colon x^{x_0,c}(t)\in X_{{++}}, \, \forall t\geq 0\}.
\end{equation}
For the reasons anticipated in the introduction and that will become clearer later, we are led to consider a larger class of admissible controls. Precisely, given $x_0 \in X$ and $\varphi^\star \in X^\star_{++}$, we define the set:
\begin{align}\label{eq:Arelax}
\mathcal{A}^{\varphi^\star}_{++}(x_0)&\coloneqq\{c\in  \dL^1_{loc}(\R_+;X_+) \colon \, \langle x^{x_0,c}(t),\varphi^\star\rangle > 0, \, \forall t\geq 0\}\\
&= \{c\in  \dL^1_{loc}(\R_+;X_+) \colon \  x^{x_0,c}(t) \in {X^{\varphi^\star}_{++}}, \, \forall t\geq 0\},\nonumber
\end{align}
%%The purpose of introducing the latter class is to relax the constraint imposed by $\eqref{eq:A++def}$ to treat a simpler problem. The fact that the constraint is really relaxed when considering $\mathcal{A}^{\varphi^\star}_{++}(x_0)$ in place of $\mathcal{A}_{{++}}(x_0)$  we can rewrite $\mathcal{A}^{\varphi^\star}_{++}(x_0)$ as
%%\begin{equation*}
%%\mathcal{A}^{\varphi^\star}_{++}(x_0)=\{c\in  \dL^1_{loc}(\R_+;X_+) \colon \  x^{x_0,c}(\cdot) \in {X^{\varphi^\star}_{++}}, \, \forall t\geq 0\},
%%\end{equation*}
where $X^{\varphi^\star}_{++}$ is the open (infinite-dimensional) half-space of $X$ generated by $\varphi^\star\in X_{++}^\star$, i.e.,
\begin{equation*}
X^{\varphi^\star}_{++}\coloneqq\{f\in  X:  \ \ \langle f,\varphi^\star\rangle >0\}.
\end{equation*}
Notice that, for any $\varphi^\star\in X_{++}^\star$, we have the set inclusion $X_{++} \subseteq X^{\varphi^\star}_{++}$. In turn, this implies $\mathcal{A}_{{++}}(x_0)\subseteq  \mathcal{A}^{\varphi^\star}_{++}(x_0)$, for any $x_0 \in X$. Thus, we are relaxing the state constraint when passing from the former to the latter set of admissible controls.
%
%Supposing that $\mathcal{A}^{\varphi^\star}_{++}(x_0)$ is not empty, if $c \in \mathcal{A}^{\varphi^\star}_{++}(x_0)$, then the state variable remains in $X^{\varphi^\star}_{++} \supset X_{++}$. Hence, controls in $\mathcal{A}^{\varphi^\star}_{++}(x_0)$ provide a relaxation of state constraint \eqref{eq:A++def}.
For future reference, we observe that: 
\begin{itemize}
\item[--] both $\mathcal{A}_{{++}}(x_0)$ and $\mathcal{A}^{\varphi^\star}_{++}(x_0)$ are convex sets, due to linearity of the state equation;
\item[--] if $x_0\in X_{++}$,  then $\mathcal{A}_{{++}}(x_0)$ and $\mathcal{A}^{\varphi^\star}_{++}(x_0)$ are non-empty, as  
 the null control $c(\cdot)\equiv 0$ belongs to them;
 \item[--] since $N$ is a positive operator, we have the following monotonicity property:
$$
c_1(t) \leq_X c_2(t), \text{ for almost all } t \geq 0 \ \Longrightarrow \ x^{x_0,c_1}(t) \geq_X  x^{x_0,c_2}(t), \text{ for all } t \geq 0;
$$
hence, if $c_1(t) \leq_X c_2(t)$ for a.e. $t \geq 0$ and $c_2\in \mathcal{A}^{\varphi^\star}_{++}(x_0)$, then $c_1\in \mathcal{A}^{\varphi^\star}_{++}(x_0)$.
\end{itemize}
}

\blu{We  complete our setting defining the functional  to optimize.
Let $u \colon D\times \R_+\to\R\cup\{-\infty\}$ be a measurable function satisfying the following assumption, that will be standing throughout the paper.}
\blu{\begin{assumption}\label{ass:u}
The function $u \colon D\times \R_+\to\R\cup\{-\infty\}$ is such that $u(\theta,\cdot)$ is increasing and concave\footnote{From an economic perspective, these two properties entail that $u(\theta, \cdot)$ is a utility function, for any $\theta \in D$.} for all $\theta \in D$. Moreover,
$u$ is either bounded from above or from below.
Without loss of generality we assume that either $u \colon D\times\R_+\to[-\infty,0]$ or  $u \colon D\times \R_+\to\R_+.$
\end{assumption}}
\blu{Next, we consider the functionals
\begin{equation*}
\mathcal U(z)\coloneqq\int_D u(\theta,z(\theta)) \, \mu(\d\theta), \ \ \ \ z\in X_+,
\end{equation*}
and
\begin{equation}\label{eq:functional}
\blu{\mathcal{J}(c)\coloneqq\int_0^\infty \de^{-\rho t}	\mathcal{U}(c(t)) \, \d t, \quad c \in \dL^1_{loc}(\R_+;X_+),}
\end{equation}
where $\rho>0$ is a given discount factor. 
%We introduce, also, the following restrictions of $\mathcal J$:
%\begin{align}\label{eq:J}
%\mathcal J(x_0; c) &\coloneqq \mathcal J(c), \quad c \in \mathcal{A}_{{++}}(x_0), \, x_0 \in X,
%\\
%\mathcal J(x_0, \varphi^\star; c) &\coloneqq \mathcal J(c), \quad c \in \mathcal{A}^{\varphi^\star}_{++}(x_0), \, x_0 \in X, \varphi^\star \in X^\star_{++}.
%\end{align}
Notice that both $\mathcal{U}$ and, consequently, $\mathcal{J}$ inherit concavity from $u$.
% Moreover, thanks to linearity of the state equation~\eqref{state} and to the fact that $X_{++}$ and $X^{\varphi^\star}_{++}$, $\varphi^\star \in X^\star_{++}$, are convex sets, also the restrictions $\mathcal J(x_0; c)$ and $\mathcal J(x_0, \varphi^\star; c)$ are concave, for any $x_0 \in X$ and any $\varphi^\star \in X^\star_{++}$.
}

\blu{
The optimal control problem we are interested in is 
\begin{equation}\label{eq:pbP}
\tag{$P$} \text{Maximize } \mathcal{J}(c) \text{ over the set } \mathcal{A}_{{++}}(x_0), \quad x_0\in {X_{++}},
\end{equation}
whose value function is
\begin{equation*}
V(x_0)\coloneqq\sup_{c\in \mathcal{A}_{++}(x_0)} \mathcal{J}(c), \quad x_0\in {X_{++}}.
\end{equation*}
As anticipated, we relax the state constraint imposed by~\eqref{eq:A++def} introducing the family, parameterized by $\varphi^\star \in {X^\star_{++}}$, of auxiliary optimal control problems
\begin{equation}\label{eq:pbPhi}
\tag{$P^{\varphi^\star}$} \text{Maximize } \mathcal{J}(c) \text{ over the set } \mathcal{A}^{\varphi^\star}_{++}(x_0), \quad x_0\in {X_{++}},
\end{equation}
whose value function is
\begin{equation*}
V^{\varphi^\star}(x_0)\coloneqq\sup_{c\in \mathcal{A}^{\varphi^\star}_{++}(x_0)} \mathcal{J}(c), \quad x_0\in {X_{++}}.
\end{equation*}
\begin{remark}\label{rem:importante}
{Since the state constraint imposed by~\eqref{eq:A++def} is stricter than the one imposed by~\eqref{eq:Arelax}, we see that}
\begin{itemize}
\item[(i)] $V^{\varphi^\star}(x_0)\geq V(x_0)$ for every $x_0\in X_{{++}}$; \medskip
\item[(ii)]
If $\hat{c}\in \mathcal{A}^{\varphi^\star}_{++}(x_0)$ is optimal for \eqref{eq:pbPhi}\! and belongs to  $\mathcal{A}_{++}(x_0)$, then it is optimal for \eqref{eq:pbP}.
\end{itemize}
\end{remark}}
\blu{
The reason to relax the state constraint and to consider the new family of optimal control problems is that, as we will show in 
 Section~\ref{sec:hjb}, problem~\eqref{eq:pbPhi} admits an explicit solution and an optimal feedback control. If it is possible to prove that this feedback control belongs to $\mathcal{A}_{++}(x_0)$, then the sufficient condition of Remark~\ref{rem:importante}~(ii) holds and, consequently, this feedback control will be optimal for problem~\eqref{eq:pbP}, too. In the economic growth model studied in Section~\ref{sec:app} we will show that this is the case (provided that suitable assumptions hold), thanks to the stability result provided by Theorem~\ref{th:stability}.
}

\blu{It is worth noticing that we cannot say \emph{ex ante} that $V^{\varphi^\star}$ is finite. Even in simple one-dimensional cases (see, e.g., \cite{FreniGozziSalvadori06}) it may be always $+\infty$ or always $-\infty$. Sufficient conditions for finiteness will be provided later (see point~\textit{(\ref{ass:lambda0})} of Assumption~\ref{ass:explicit} and point~\textit{(\ref{rem:Vfinite})} of Remark~\ref{rem:expl}).}

\blu{We conclude this section introducing the \emph{Hamilton-Jacobi-Bellman} equation (HJB, {for short}) associated to the optimal control problem \eqref{eq:pbPhi} indexed by $\varphi^\star \in {X^\star_{++}}$:
\begin{equation}\label{HJB}
\rho v(x)=\langle Lx,\nabla v(x)\rangle + \mathcal{H}(\nabla v(x)), \ \ \ \ x\in X_{++}^{\varphi^\star},
\end{equation}
where the Hamiltonian function $\mathcal{H}$ is defined by
$$
\mathcal{H}(q^\star) \coloneqq\sup_{z\in X_+}\mathcal{H}_{CV}(q^\star;z),  \qquad q^\star\in X^\star,
$$
and $\mathcal{H}_{CV}$ is  the Current Value Hamiltonian function
$$
\mathcal{H}_{CV}(q^\star;z) \coloneqq\mathcal{U}(z)-\langle Nz,q^\star\rangle, \quad z\in X_+,\, q^\star\in X^\star.
$$
Notice that, without further assumptions, it may happen that $\mathcal H(q^\star) = +\infty$, for some $q^\star \in X^\star$.
Finiteness of $\mathcal H(q^\star)$ is guaranteed at least in the following cases:
\begin{itemize}
\item[--] If $\calu$ is bounded above and $q^\star\in X^\star_+$;\smallskip
\item[--] If $\calu$ is bounded below and there exist constants $C > 0$, $\alpha \in (0,1)$, and $k > 0$, such that $\calu(z) \leq C |z|_X^\alpha$ and $\langle Nz, q^\star\rangle \geq k |z|_X$, $q^\star \in X^\star$.
\end{itemize}
}
%\begin{definition}
%We call \emph{classical solution} to \eqref{HJB} (on  $X_{++}^{\varphi^\star}$) a function
%$v\in \mathcal C^1(X_{++}^{\varphi^\star} ;\R)$ such that
%%\footnote{This equality, in particular, implies that $\mathcal{H}(\nabla v(x))$ is finite for every $x\in X_{++}^{\varphi^\star}$.}
%$\nabla v\in \mathcal C( X_{++}^{\varphi^\star} ;D(L^\star))$ and such that
%$$
%\rho v(x)=\langle x,L^\star\nabla v(x)\rangle + \mathcal{H}(\nabla v(x)), \ \ \ \ \forall x\in X_{++}^{\varphi^\star}.
%$$
%This, in particular, implies that $\mathcal{H}(\nabla v(x))$ is finite for every $x\in X_{++}^{\varphi^\star}$.
%\end{definition}

\blu{In this paper we will consider classical solutions to HJB equation~\eqref{HJB}, according to the following definition.} 
\blu{\begin{definition}\label{def:solHJB}
{A function
$v\in \mathcal C^1(X_{++}^{\varphi^\star} ;\R)$ such that $\nabla v\in \mathcal C( X_{++}^{\varphi^\star} ;D(L^\star))$ is called a \emph{classical solution} to \eqref{HJB} on  $X_{++}^{\varphi^\star}$ if 
\begin{equation*}
\rho v(x)=\langle x,L^\star\nabla v(x)\rangle + \mathcal{H}(\nabla v(x)), \ \ \ \ \forall x\in X_{++}^{\varphi^\star}.
\end{equation*}}
\end{definition}
\blu{We point out that this notion of solution for an HJB equation in infinite dimension is very demanding, because of the required regularity of the solution itself. Nonetheless, we will see in Section~\ref{sec:hjbexpl} that, under suitable assumptions, it is possible to find explicit solutions that verify the definition above. In more general cases, one can resort to other (weaker) notions, such as viscosity solutions (see Section~\ref{sec:conclusion} for comments and references on this point).}
\begin{remark} Notice that, if $v$ is a classical solution to HJB equation~\eqref{HJB}, then $\mathcal{H}(\nabla v(x))$ must be finite for every $x\in X_{++}^{\varphi^\star}$.
\end{remark}}

\section{Verification theorem, explicit solutions, and stability}\label{sec:hjb}
\blu{In this section we  focus  on the family of optimal control problems~\eqref{eq:pbPhi}. First, for each fixed $\varphi^\star \in {X^\star_{++}}$, we will provide a verification theorem in the general setting presented in Section~\ref{sec:model}. Then, specializing our setting and suitably choosing $\varphi^\star \in {X^\star_{++}}$, we will provide an explicit solution to problem~\eqref{eq:pbPhi} and  give a stability result for this solution. These results will be crucial to study our motivating economic application in Section~\ref{sec:app}.}

\subsection{Verification Theorem}\label{sec:ver}
{Typically, to prove a verification theorem for infinite horizon problems, a condition on the solution $v$ computed on the admissible trajectories when $t \to + \infty$ is needed.
%This is exactly the analogous of the so-called transversality condition arising in the maximum principle approach.
\blu{Given $x_0 \in X$ and $\varphi^\star \in {X^\star_{++}}$, the (relaxed on integer numbers) condition that we shall use is}}
\begin{equation}\label{eq:trasvVFnew}
\lim_{k\in\mathbb{N},\  k \to +\infty} \de^{-\rho k} v(x^{x_0,c}(k))=0,
\quad \forall  c\in \mathcal{A}^{\varphi^\star}_{++}(x_0) \text{ s.t. } \blu{\mathcal{J}(c)}>-\infty,
\end{equation}

\begin{theorem}[Verification]\label{verification}
{Let $\varphi^\star \in {X^\star_{++}}$ and $x_0\in X_{++}^{\varphi^\star}$}. Let $v$ be a classical solution to \eqref{HJB} and {assume that} \eqref{eq:trasvVFnew} holds.
Then:
\begin{itemize}
\item[(i)]
$v(x_0)\geq V^{\varphi^\star}(x_0)$;\smallskip
\item[(ii)] If, moreover, there exists
$\hat c\in\mathcal{A}^{\varphi^\star}_{++}(x_0)$ such that, for a.e. $s \ge 0$,
\begin{equation}\label{argmax}
\mathcal{H}\left(\nabla v(x^{x_0,\hat c}(s))\right)
=\mathcal{H}_{CV}\left(\nabla v(x^{x_0,\hat c}(s));
\hat c(s)\right)
\quad\Longleftrightarrow\quad
N^\star\nabla v(x^{x_0,\hat c}(s))\in D^+\mathcal{U}(\hat c(s)), 
%\ \ \ \mbox{for a.e.} \ s\geq  0,
\end{equation}
where $D^+\mathcal{U}$ denotes the superdifferential of $\mathcal{U}$,
then $v(x_0)=V^{\varphi^\star}(x_0)$ and $\hat c$ is optimal for \eqref{eq:pbPhi} starting at $x_0$, i.e., $\blu{\mathcal{J}(\hat c)}=V^{\varphi^\star}(x_0)$.
\end{itemize}
\end{theorem}

\begin{proof} (i)
Let $c\in\mathcal{A}^{\varphi^\star}_{++}(x_0)$ be such that $\blu{\mathcal{J}(c)}>-\infty$. By chain's rule in infinite dimension (see \cite{LY}), we have, for every $t\geq 0$,
$$
\frac{\d}{\d t} \big[\de^{-\rho t}v(x^{x_0,c}(t)))	\big]=
\de^{-\rho t}\big(-\rho  v(x^{x_0,c}(t))+\langle x^{x_0,c}(t),L^\star\nabla v(x^{x_0,c}(t)) \rangle-\langle Nc(t),\nabla v(x^{x_0,c}(t))\rangle  \big).
$$
Now we add  and subtract $\de^{-\rho t}\mathcal{U}(c(t))$ to the right hand side, use the fact that $v$ solves HJB, and integrate over $[0,t]$. We get, for every $t\geq 0$,
%\begin{align*}
%e^{-\rho t} v(x^{x_0,c}(t))-v(x_0)&=-\int_0^t e^{-\rho s}\mathcal{U}(c(s))\d s\\
%&+ \int_0^t e^{-\rho s} \bigg(-\mathcal{H}(\nabla v(x^{x_0,c}(s))+\mathcal{U}(c(s))-\langle Nc(s),\nabla v(x^{x_0,c}(s))\rangle\bigg) \d s,
%\end{align*}
%i.e.
\begin{align*}
& \de^{-\rho t} v(x^{x_0,c}(t))+\int_0^t \de^{-\rho s}\mathcal{U}(c(s)) \, \d s
=
\\
& v(x_0)
+ \int_0^t \de^{-\rho s} \Big(-\mathcal{H}(\nabla v(x^{x_0,c}(s))+\mathcal{H}_{CV}(\nabla v(x^{x_0,c}(s);c(s))\Big) \, \d s,
\end{align*}
Observe that, since $\calu$ is concave (hence, sublinear from above) and $c(\cdot)\in \dL^1_{loc}(\R_+;X_+)$,
both sides of the above inequality are finite for every $t\ge 0$.
Rearranging the terms and taking into account \eqref{argmax} and
the definition of $\mathcal{H}$, 
we get, for every $t\geq 0$,
\begin{equation}\label{FEineq}
v(x_0) \geq \de^{-\rho t} v(x^{x_0,c}(t))+\int_0^t \de^{-\rho s} \mathcal{U}(c(s)) \, \d s.
\end{equation}
Since the sign of $\mathcal{U}$ is constant (cf. Assumption~\ref{ass:u}), we have that
\begin{equation}\label{uu}
\lim_{k\in\mathbb{N}, \ k \to +\infty}\int_0^k \de^{-\rho s} \mathcal{U}(c(s)) \, \d s=
\int_0^{\infty} \de^{-\rho s} \mathcal{U}(c(s)) \, \d s =
\blu{\mathcal{J}(c)}.
\end{equation}
Hence, passing to the limit in~\eqref{FEineq}, as $k\to\infty$, $k \in \mathbb N$, and using~\eqref{eq:trasvVFnew}, we deduce that
$$
v(x_0)\ge\blu{ \mathcal{J}(c)}.
$$
Then, given the arbitrariness of $c\in\mathcal{A}^{\varphi^\star}_{++}(x_0)$ such that $\blu{\mathcal{J}(c)}>-\infty$ and by definition of $V^{\varphi^\star}$, we immediately get the claim.\smallskip

(ii)
Notice that, by concavity of $\mathcal{U}$, \eqref{argmax} is equivalent to
\begin{equation}\label{argmaxbis}
\hat c(s)\in \argmax_{z\in X^{\varphi^\star}_{++}} \big\{\,\mathcal{U}(z)-\left\langle Nz,\nabla v(x^{x_0,\hat c}(s))\right\rangle\big\}, \ \ \ \mbox{for a.e.} \  s\geq 0,
\end{equation}
the usual closed loop condition for optimality.
Hence, for $c=\hat c$ we have equality in \eqref{FEineq} and therefore, passing to the limit, as $k\to\infty$, $k \in \mathbb N$, and using~\eqref{eq:trasvVFnew} and~\eqref{uu}, we get the equality
$$
v(x_0)= \blu{\mathcal{J}(\hat c)}.
$$
Since $\blu{\mathcal{J}(\hat c)}\le V^{\varphi^\star}(x_0)$ and since that the reverse inequality holds by part (i) of the theorem, the claim follows.
\end{proof}

\subsection{Explicit solutions to HJB equation and optimal feedback control}\label{sec:hjbexpl}
%For these two cases we state two distinct sets of assumptions.
\blu{In this section we are going to provide an explicit solution to the HJB equation~\eqref{HJB} for problem~\eqref{eq:pbPhi}, for a specific choice of $\varphi^\star \in X^\star_{++}$. To this end, we need to specialize the setting of Section~\ref{sec:model} by introducing the following assumption that will be in force throughout this section.}
%First, we make explicit the linear operator $N$ of the state equation~\eqref{state} and the function $u$ of Assumption~\ref{ass:u}. 
\begin{assumption}\label{ass:explicit}\blu{ $X$ is either the space $\dL^p(D, \mu)$, $p \in [1,+\infty)$, where $\mu$ is a $\sigma$-finite measure, or the space $\calc(D)$ of real-valued continuous functions on a compact metric space $D$, equipped with the sup-norm and with a Borel measure $\mu$. }

\blu{
Moreover, the linear operator $N$ in~\eqref{state} and the function $u$ of Assumption~\ref{ass:u} are explicitly given by
\begin{align}
[Nz](\theta)&=\eta(\theta)z(\theta),  \quad z \in X, \, \theta \in D, \label{eq:Nexpl}\\
u(\theta,\xi)&=\displaystyle{\frac{\xi^{1-\gamma}}{1-\gamma} f(\theta)},  \quad \theta \in D, \, \xi \in \R_+, \label{eq:uexpl}
\end{align}
where 
$\eta, f: D\to (0,+\infty)$ are measurable functions and $\gamma\in (0,1)\cup(1,+\infty)$ is a fixed parameter.}

\blu{Furthermore:
\begin{enumerate}[(i)]
\item\label{ass:eigenvect}
There exists an eigenvector ${b_0^\star}\in X^\star_{++}$ for  $L^\star \colon D(L^\star)\subseteq X^\star\to X^\star$  with eigenvalue $\lambda_0^\star\in \R$.
%\item The strictly positive measure ${b_0^\star}$ is absolutely continuous with respect to $\mu$. Morever, still denoting by ${b_0^\star}$ its density, we assume that ${b_0^\star} \in \calc(D; (0,+\infty))$.
\item\label{ass:lambda0} $\rho>\lambda_0^\star(1-\gamma)$, where $\rho$ is the discount factor appearing in~\eqref{eq:functional}.
\item\label{ass:gamma>1} If $\gamma>1$, then  $\frac{({b_0^\star})^{1-\gamma}}{f}\in \dL^\infty(D,\mu;\R_+)$.
\item\label{ass:C} If $X = \calc(D)$, the strictly positive measure ${b_0^\star}$ is absolutely continuous with respect to $\mu$, with density still denoted by ${b_0^\star}$. In addition, $b_0^\star, \eta, f \in \calc(D; (0,+\infty))$. 
%\red{(A me sembra che serva $f > 0$, vd. la dimostrazione della Proposition~\ref{prop:exp}. Se così è, la frase "then $f=0$ on a set of zero measure $\mu$" del punto precedente è ridondante)}
\item\label{ass:Lp} If $X = \dL^p(D,\mu)$, $\eta, f \in \dL^\infty(D,\mu;(0,+\infty))$ and, moreover,
\begin{equation}\label{eq:integr}
\int_D f(\theta)^{\frac{1}{\gamma}}(\eta(\theta){b_0^\star}(\theta))^{\frac{\gamma-1}{\gamma}} \mu(\d\theta)<\infty, \ \  \int_D\left(\frac{f(\theta)}{\eta(\theta){b_0^\star}(\theta)}\right)^{p/\gamma}\mu(\d\theta)<\infty.
\end{equation}
\end{enumerate}
}
\end{assumption}
\blu{
\begin{remark}\label{rem:expl}
Let us comment on the specific setting described above.
\begin{enumerate}[(i)]
\item It is crucial to assume the existence of the strictly positive eigenvector $b_0^\star \in X^\star_{++}$, as required by point~(\ref{ass:eigenvect}) of Assumption~\ref{ass:explicit}. Indeed, {this enables us to find an explicit solution to  HJB \eqref{HJB} associated to the auxiliary problem ($P^{b_0^\star}$) in the half-space $X^{b_0^\star}_{++}$.
\item In the case $X = \calc(D)$ with $D$ compact space, it is important that $b_0^\star$ is represented as a continuous function to guarantee well-posedness of the feedback operator $\Phi$ defined in \eqref{def:phi}, whence point~(\ref{ass:C}) of Assumption~\ref{ass:explicit}.}
\item\label{rem:Vfinite} {Point~(\ref{ass:lambda0}) of Assumption~\ref{ass:explicit}} is needed to ensure finiteness of the solution of the HJB equation, and hence of the value function.
\item {Point~(\ref{ass:gamma>1}) of Assumption~\ref{ass:explicit}} is required to verify \eqref{eq:trasvVFnew} in Lemma~\ref{lm:trasv}.
\item\label{rem:etauhyp} {Given the expression of $u$ and since $\eta$ and $f$, appearing in~\eqref{eq:Nexpl} and~\eqref{eq:uexpl}, respectively, are positive, point~\textit{(\ref{hyp:N})} of Assumption~\ref{ass:L} and Assumption~\ref{ass:u} are verified.}
\item {In the case $X=\dL^p(D,\mu)$, the functions $\eta$ and $f$} do not need to be continuous, but only essentially bounded. However, additional integrability conditions must hold: the first requirement of~\eqref{eq:integr} allows to give sense to the explicit solution of HJB; the second condition of~\eqref{eq:integr} is used to make sense of the optimal feedback map. Both conditions are automatically verified when $X = \calc(D)$ with $D$ compact, due to compactness of $D$.
\end{enumerate}
\end{remark}}
%
%For the rest of the section we will assume that  Assumption~\ref{ass:explicit} (in the case $X=\mathcal{C}(D)$) {or Assumption~\ref{ass:explicitLp}  (in the case $X=\dL^p(D,\mu)$) hold.}
\blu{We are now ready to provide the explicit solution to~\eqref{HJB} for the auxiliary problem~($P^{b_0^\star}$).}
\begin{proposition}\label{prop:exp}
  The function
\begin{equation}\label{solution}
v(x) \coloneqq \alpha\frac{\langle x,{b_0^\star}\rangle^{1-\gamma}}{1-\gamma}, \quad x\in X^{b_0^\star}_{++},
\end{equation}
where
\begin{equation}\label{alpha}
\alpha \coloneqq \gamma^\gamma\left(\frac{\int_D f(\theta)^{\frac{1}{\gamma}}(\eta(\theta){b_0^\star}(\theta))^{\frac{\gamma-1}{\gamma}} \mu(\d\theta)}{\rho-\lambda^\star_0(1-\gamma)}\right)^\gamma,
\end{equation}
is a classical solution to~\eqref{HJB}.
\end{proposition}
\begin{proof}
\blu{We point out, first, that the arguments provided here are the same regardless of the choice of $X$ in Assumption~\ref{ass:explicit}. }

\blu{We start noticing that $v$ is differentiable, as composition of a linear form with a power function, and that
\begin{equation*}
\nabla v(x)=\alpha \langle x,{b_0^\star}\rangle ^{-\gamma} {b_0^\star}, \quad x\in X^{b_0^\star}_{++}.
\end{equation*}
Due to Assumption~\ref{ass:explicit}~\textit{(\ref{ass:eigenvect})}, and since $\alpha > 0$ by point~\textit{(\ref{ass:lambda0})} and either point~\textit{(\ref{ass:C})} or point~\textit{(\ref{ass:Lp})} of Assumption~\ref{ass:explicit}, $v$ enjoys the regularity required by Definition~\ref{def:solHJB}.
Next, observe that by~\eqref{eq:uexpl} and by Assumption~\ref{ass:explicit} --- either point~\textit{(\ref{ass:C})} or point~\textit{(\ref{ass:Lp})} ---, the functional $\mathcal{U}$ is Fr\'echet differentiable and
\begin{equation}\label{eq:gradU}
[\nabla \mathcal{U}(z)](\theta)=f(\theta) z(\theta)^{-\gamma}, \quad z\in X_+, \, \theta \in D.
\end{equation}
The adjoint of the operator $N$, given in~\eqref{eq:Nexpl}, is
\begin{align*}
&N^\star q^\star(\theta) = \eta(\theta) q^\star(\theta), \quad \theta \in D, & &\text{when } X = \dL^p(D,\mu),
\\
&N^\star q^\star(\d\theta) = \eta(\theta) q^\star(\d\theta), \quad \theta \in D, & &\text{when } X = \cC(D).
\end{align*}
Therefore, taking into account point~\textit{(\ref{ass:C})} of Assumption~\ref{ass:explicit}, we get (regardless of the choice of $X$)
\begin{equation}\label{eq:Nstar}
N^\star \nabla v(x) = \alpha \langle x,{b_0^\star}\rangle ^{-\gamma} \eta(\cdot) {b_0^\star}, \quad x\in X^{b_0^\star}_{++}.
\end{equation}
Putting together~\eqref{eq:gradU} and~\eqref{eq:Nstar}, we have
\begin{equation*}
\frac{\partial \mathcal{H}_{CV}}{\partial z} (\nabla v(x);z)= \nabla \mathcal{U}(z)-N^{\star}\nabla v(x)= f(\cdot)z^{-\gamma}- \alpha \langle x,{b_0^\star}\rangle^{-\gamma} \eta(\cdot) b_0^\star(\cdot), \quad x\in X^{b_0^\star}_{++}, \quad z \in X_+.
\end{equation*}
Since $z\mapsto \mathcal{H}_{CV}(\nabla v(x);z)$ is strictly concave, for each fixed $x\in X^{b_0^\star}_{++}$,
the unique maximum point $\hat{z}(x)$  of this function is provided by 
\begin{equation}\label{optimizer}
\hat{z}(x)=\argmax_{z\in X_{+}}\ \mathcal{H}_{CV}(\nabla v(x);z)=\left(\frac{f(\cdot)}{\alpha \eta(\cdot)b_0^\star(\cdot)}\right)^{\frac{1}{\gamma}}\langle x,{b_0^\star}\rangle, \quad x\in X^{b_0^\star}_{++}, 
\end{equation}
%$$(fz^{-\gamma})(\theta)=(N^\star \nabla v(x))(\theta)= \alpha \langle x,{b_0^\star}\rangle ^{-\gamma} \eta(\theta) {b_0^\star}(\theta), \ \ \ \ \forall \theta\in D.$$
%By routine computations we get, for all $x\in X^{b_0^\star}_{++}$,
whence
\begin{equation*}
\mathcal{H}(\nabla v(x)) = \mathcal{H}_{CV}(\nabla v(x);\hat{z}(x)) =	\frac{\gamma \langle x,{b_0^\star}\rangle^{1-\gamma}}{1-\gamma}\int_D f(\theta)^{\frac{1}{\gamma}}(\alpha \eta(\theta)b_0^\star(\theta))^{\frac{\gamma-1}{\gamma}} \mu(\d\theta), \quad x\in X^{b_0^\star}_{++}.
\end{equation*}
%with optimizer
%\begin{equation}\label{optimizer}
%\hat{z}(\nabla v(x))=\argmax_{z\in X_{+}}\mathcal{H}_{CV}(\nabla v(x);z)=\left(\frac{f(\cdot)}{\alpha \eta(\cdot)b_0^\star(\cdot)}\right)^{\frac{1}{\gamma}}\langle x,{b_0^\star}\rangle .
%\end{equation}
}
\blu{Plugging the latter expression into \eqref{HJB}, we get the algebraic equation in $\alpha$
$$\frac{\rho}{1-\gamma} \alpha =\lambda_0^\star \alpha+\frac{\gamma}{(1-\gamma)}\alpha^{\frac{\gamma-1}{\gamma}}\int_D f(\theta)^{\frac{1}{\gamma}}(\eta(\theta){b_0^\star}(\theta))^{\frac{\gamma-1}{\gamma}} \mu(\d\theta),$$
which has a unique positive solution  provided by \eqref{alpha}.}
\end{proof}

\blu{\begin{remark}
\label{rem:b0changes}
It is easy to verify that the function $v$ does not change if we multiply $b_0^\star$ by any strictly positive constant.  This is consistent with the fact that, if $k > 0$ and  $\varphi^\star = k b_0^\star$, then $X^{\varphi^\star}_{++}=X^{b_0^\star}_{++}$; hence, $\cala^{\varphi^\star}_{++}(x_0)=\cala^{b_{0}^\star}_{++}(x_0)$.
\end{remark}}

\blu{The following map provides the unique maximum  point given in \eqref{optimizer} as a function of the state:
\begin{equation}\label{def:phi}
[\Phi x](\theta)=\hat{z}(x)(\theta)=
\left(
\frac{f(\theta)}{\alpha \eta(\theta){b_0^\star}(\theta)}
\right)^{\frac{1}{\gamma}}
\langle x, {b_0^\star}\rangle, \quad x \in X, \, \theta \in D.
\end{equation}
It is important to notice that $\Phi \colon X \to X$ is a bounded and positive linear operator. 
This is the natural candidate to be the optimal feedback map for problem $(P^{b_0^\star})$, i.e., the map that will provide (as we will show) the optimal feedback control for this problem, given any initial condition $x_0\in X_{++}$.
Therefore, we introduce the \emph{closed loop equation} for problem $(P^{b_0^\star})$:
\begin{equation}\label{cle}
\left\{
\begin{aligned}
&x'(t)=Lx(t)-N\Phi x(t), \quad t \geq 0, \\
&x(0)=x_0 \blu{\, \in X}.
\end{aligned}
\right.
\end{equation}
This linear equation admits a unique mild solution, {that will be denoted by} $\hat{x}^{x_0}(\cdot)$, which is also a weak solution (see \cite[p. 129-131]{BDDM}). 
Notice that the operator $N\Phi \colon X\to X$, {appearing} in \eqref{cle}, is {explicitly} given by:
\begin{equation}\label{def:Nphi}
[N\Phi x](\theta)=
\eta(\theta)\left(
\frac{f(\theta)}{\alpha \eta(\theta){b_0^\star}(\theta)}
\right)^{\frac{1}{\gamma}}
\langle x, {b_0^\star}\rangle, \quad x \in X, \, \theta \in D.
\end{equation}}
\blu{To study \eqref{cle}, let us} define the operator
\begin{equation}\label{eq:Bdef}
B \coloneqq L- N\Phi.
\end{equation}
Since $B$ is a bounded perturbation of $L$, by \citep[Ch. III, Theorem~1.3]{engelnagel:sgr} we have that it is a closed operator with domain $D(B)=D(L)$ and {that it} generates a $C_0$-semigroup $\{\de^{tB}\}_{t\geq 0}$ on $X$.
\blu{The next lemma will be useful in the sequel.}
\begin{lemma} 
\label{lm:gautov} We have the following.
\begin{enumerate}[(i)]
\item[]
\item[(i)] The adjoint operator 
$(N\Phi)^\star:X^{\star}\to X^{\star}$  is explicitly given by:
{\allowdisplaybreaks
\begin{align}
\label{def:Nphiadjoint} (N\Phi)^\star q &=\alpha^{-\frac{1}{\gamma}}
\left(\int_{D}f(\theta)^{\frac{1}{\gamma}}
\eta(\theta)^{1-\frac{1}{\gamma}}
b_0^\star(\theta)^{-\frac{1}{\gamma}}q(\d\theta)
\right)b_0^\star, & &\text{when } X=\mathcal{C}(D),
\\
\label{def:Nphiadjointbis}
(N\Phi)^\star q &=\alpha^{-\frac{1}{\gamma}}
\left(\int_{D}f(\theta)^{\frac{1}{\gamma}}
\eta(\theta)^{1-\frac{1}{\gamma}}
b_0^\star(\theta)^{-\frac{1}{\gamma}}q(\theta)\mu(\d\theta)
\right)b_0^\star, & &\text{when } X=\dL^p(D,\mu).
\end{align}
}
\item[(ii)]
Define
\begin{equation}\label{g}
g\coloneqq\frac{\lambda_0^\star-\rho}{\gamma}.
\end{equation}
Then $g$ is an eigenvalue of $B^\star$ and
$b_0^\star$ is an eigenvector of $B^\star$ associated to $g$.
\end{enumerate}
\end{lemma}

\begin{proof}
The proof of (i) follows from straightforward computations.
For the proof of (ii) we observe that, using (i), \eqref{alpha}, and \eqref{g}, we get 
\begin{equation*}
B^\star b_0^\star =L^\star b_0^\star - (N\Phi)^\star b_0^\star
=\lambda_0^\star b_0^\star -\frac{\rho-\lambda_0^\star(1-\gamma)}{\gamma}
b_0^\star =
\left[\lambda_0^\star-\frac{\rho-\lambda_0^\star(1-\gamma)}{\gamma}\right]
b_0^\star=gb_0^\star. \qedhere
\end{equation*}
\end{proof}

\begin{lemma}
\label{lm:gautovdyn}
Let $x_0\in X_{++}^{{b_0^\star}}$ and let $\hat{x}^{x_0}$ be the unique mild (weak) solution to \eqref{cle}. Then,
%\begin{align*}
%\frac{\d}{\d t}	\langle  \hat{x}^{x_0}(t),{b_0^\star}\rangle &=g\,\langle \hat{x}^{x_0}(t),{b_0^\star}\rangle,
%\ \ \ \ \ \forall t \geq 0,\end{align*}
%i.e.
\begin{equation}\label{eq:hatx0expl}
\langle  \hat{x}^{x_0}(t),{b_0^\star}\rangle=\langle x_0,{b_0^\star}\rangle\,  \de^{gt}, \ \ \ \forall t\geq0,
\end{equation}
where $g$ is given by \eqref{g}.
Hence, in particular,
\begin{equation}\label{admissibility}
\hat{x}^{x_0}(t)\in X_{++}^{{b_0^\star}}, \ \ \ \ \ \forall t \geq 0.
\end{equation}
\end{lemma}
\begin{proof}
Testing the closed loop equation \eqref{cle} against ${b_0^\star}\in D(L^\star)$, by Lemma~\ref{lm:gautov} we get
\begin{equation*}
\frac{\d}{\d t}	\langle  \hat{x}^{x_0}(t),{b_0^\star}\rangle =\langle \hat{x}^{x_0}(t),L^\star{b_0^\star}\rangle-\langle N\Phi\hat{x}^{x_0}(t),{b_0^\star}\rangle
=\langle \hat{x}^{x_0}(t),(L^\star-(N\Phi)^*){b_0^\star}\rangle
= g\langle \hat{x}^{x_0}(t),{b_0^\star}\rangle.
\end{equation*}
{The result follows by integrating this equation over $[0,t]$ and noticing that $\langle x_0,{b_0^\star}\rangle > 0$.}
\end{proof}

\begin{Lemma}
\label{lm:trasv}
Let $x_0\in X^{{b_0^\star}}_{++}$ and let $v$, defined in \eqref{solution}, be the solution to \eqref{HJB} associated to $(P^{b_0^\star})$. Then,
 \begin{equation}\label{eq:trasvVFnew2}
\lim_{k\in\mathbb{N},\, k \to +\infty} \de^{-\rho k} v(x^{x_0,c}(k))=0,
\ \ \ \ \forall  c\in \mathcal{A}^{{b_0^\star}}_{++}(x_0) \text{ s.t. } \ \blu{\mathcal{J}(c)}>-\infty.
\end{equation}
\end{Lemma}
\begin{proof}
\emph{Case 1: $\gamma\in(0,1)$.}  In this case $v$ is nonnegative {and this allows to prove a stronger result, namely}
\begin{equation}\label{qqq}
\lim_{t \to +\infty} \de^{-\rho t} v(x^{x_0,c}(t))=0,
\ \ \ \ \forall  c\in \mathcal{A}^{{b_0^\star}}_{++}(x_0).
\end{equation}
{Recalling that $x^{x_0,c}$ is a weak solution to the state equation~\eqref{state}, we can rewrite~\eqref{weak} with ${\varphi^\star}={b_0^\star}$ and obtain}
$$
\frac{\d}{\d t}\langle x^{x_0,c}(t),{b_0^\star}\rangle= \lambda_0^\star\langle x^{x_0,c}(t),{b_0^\star}\rangle-\langle N c(t),{b_0^\star}\rangle, \quad \forall c\in \mathcal{A}^{{b_0^\star}}_{++}(x_0), \quad {\forall t \geq 0,}
$$
i.e.,
$$\langle x^{x_0,c}(t),{b_0^\star}\rangle=\langle x_0,{b_0^\star}\rangle \de^{\lambda_0^\star t}-\int_0^t \de^{\lambda_0^\star(t-s)} \langle Nc(s), {b_0^\star}\rangle \, \d s, \quad \forall c\in \mathcal{A}^{{b_0^\star}}_{++}(x_0), \quad \blu{\forall t \geq 0.}$$
{Since $N$ is a positive operator and ${b_0^\star} \in X^\star_{++}$ (cf. point~\textit{(\ref{hyp:N})} of Assumption~\ref{ass:L} and point~\textit{(\ref{ass:eigenvect})} of Assumption~\ref{ass:explicit})}, we have
$$0\leq \langle x^{x_0,c}(t),{b_0^\star}\rangle\leq \langle x_0,{b_0^\star}\rangle \de^{\lambda_0^\star t},  \quad \forall c\in \mathcal{A}^{{b_0^\star}}_{++}(x_0), \quad {\forall t \geq 0.}$$
Hence, {the following inequality holds, for all $t \geq 0$,}
\begin{equation*}\label{eq:trasvVF}
0\leq \de^{-\rho t} v(x^{x_0,c}(t))=  \displaystyle{\alpha \de^{-\rho t}  \frac{\langle x^{x_0,c}(t),{b_0^\star}\rangle^{1-\gamma}}{1-\gamma}}
\leq \frac{\alpha}{1-\gamma} \langle x_0,{b_0^\star}\rangle \de^{-(\rho-\lambda_0^\star(1-\gamma)) t}   , \  \ \forall  c\in \mathcal{A}^{{b_0^\star}}_{++}(x_0).
\end{equation*}
{Noticing that $\langle x_0,{b_0^\star}\rangle > 0$ and} using point~\textit{(\ref{ass:lambda0})} of Assumption~\ref{ass:explicit}, we get \eqref{qqq}.

\emph{Case 2: $\gamma>1$.} In this case $v$ is nonpositive. Hence, for every $c\in \mathcal{A}^{{b_0^\star}}_{++}(x_0)$, we have that
$$
\liminf_{k\in\mathbb{N}, \ k \to +\infty} \de^{-\rho k} v(x^{x_0,c}(k))\leq \limsup_{k\in\mathbb{N}, \ k \to +\infty} \de^{-\rho k} v(x^{x_0,c}(k))\leq 0.
$$
Let $k\in\mathbb{N}$.
Then, by \eqref{weak} with ${\varphi^\star}={b_0^\star}$ and since $c\in\mathcal{A}^{{b_0^\star}}_{++}(x_0)$,
{\allowdisplaybreaks
\begin{align*}
0&\leq \langle x^{x_0,c}(k+1),{b_0^\star}\rangle = \langle x^{x_0,c}(k),{b_0^\star}\rangle+\int_{k}^{k+1}\lambda_0^\star\langle x^{x_0,c}(s),{b_0^\star}\rangle \, \d s-\int_{k}^{k+1}\langle c(s),{b_0^\star}\rangle \, \d s \\
&
= \langle x^{x_0,c}(k),{b_0^\star}\rangle \de^{\lambda_0^\star}-\int_{k}^{k+1}\de^{\lambda_0^\star(k+1-s)}\langle c(s),{b_0^\star}\rangle \, \d s.
\end{align*}
} It follows that
\begin{equation*}
\de^{-|\lambda_0^\star|}\int_{k}^{k+1}\langle c(s),{b_0^\star}\rangle \, \d s\leq \langle x^{x_0,c}(k),{b_0^\star}\rangle  \de^{\lambda_0^\star};
\end{equation*}
therefore,
\begin{equation*}
\int_{k}^{k+1}\langle c(s),{b_0^\star}\rangle \, \d s\leq \langle x^{x_0,c}(k),{b_0^\star}\rangle \de^{2|\lambda_0^\star|}.
\end{equation*}
Hence, by Jensen's inequality, monotonicity, concavity, and nonpositivity of the map $\xi\mapsto \frac{\xi^{1-\gamma}}{1-\gamma}$,
\begin{equation*}
\int_{k}^{k+1}\frac{\langle c(s),{b_0^\star}\rangle^{1-\gamma}}{1-\gamma} \, \d s\leq \frac{1}{1-\gamma}\left(\int_{k}^{k+1}\langle c(s),{b_0^\star}\rangle \, \d s\right)^{1-\gamma}\leq \frac{\langle x^{x_0,c}(k),{b_0^\star}\rangle^{1-\gamma}}{1-\gamma}\de^{2|\lambda_0^\star|(1-\gamma)}\leq 0.
\end{equation*}
Therefore,  multiplying by $\de^{-\rho(k+1)}$ and {recalling that $\rho>0$,}
%(if $\rho\leq 0$, we need to multiply by $\de^{-\rho k}$ to let the subsequent estimate work),
\begin{align}\label{aas}
c_k:=\int_{k}^{k+1}\de^{-\rho (k+1)}\frac{\langle c(s),{b_0^\star}\rangle^{1-\gamma}}{1-\gamma} \, \d s\leq \de^{-\rho k} \frac{\langle x^{x_0,c}(k),{b_0^\star}\rangle^{1-\gamma}}{1-\gamma}\de^{2|\lambda_0^\star|(1-\gamma)-\rho}\leq 0.
\end{align}
{By point~\textit{(\ref{ass:gamma>1})} of Assumption~\ref{ass:explicit},} it follows that 
{\allowdisplaybreaks
\begin{align*}
-\infty&<\left|\frac{({b_0^\star})^{1-\gamma}}{f}\right|_\infty \blu{\mathcal{J}(c)}
=\left|\frac{({b_0^\star})^{1-\gamma}}{f}\right|_\infty
\int_0^\infty \de^{-\rho s}\left(\int_D  \frac{c(s,\theta)^{1-\gamma}}{1-\gamma}	{f(\theta)}\d\theta \right) \d s
\\
&\leq
\int_0^\infty \de^{-\rho s}\left(\int_D  \frac{c(s,\theta)^{1-\gamma}}{1-\gamma}	\frac{ {b_0^\star}(\theta)^{1-\gamma}}{f(\theta)}{f(\theta)}\d\theta \right) \d s
=
\int_0^\infty \de^{-\rho s}\left(\int_D  \frac{c(s,\theta)^{1-\gamma} {b_0^\star}(\theta)^{1-\gamma}}{1-\gamma}\d\theta \right) \d s\\
&=\int_0^\infty \de^{-\rho s}\frac{\langle c(s),{b_0^\star}\rangle^{1-\gamma}}{1-\gamma} \, \d s
=\sum_{k=0}^\infty \int_k^{k+1} \de^{-\rho s}\frac{\langle c(s),{b_0^\star}\rangle^{1-\gamma}}{1-\gamma} \, \d s\leq \sum_{k=0}^\infty c_k\leq 0.
\end{align*}
}
Hence, $c_k\to 0$. Combining {this fact} with \eqref{aas}, we get
$$
0\geq \liminf_{k\in\mathbb{N}, \ k\to\infty}
\de^{-\rho k} v (x^{x_0,c}(k)) =\liminf_{k\in\mathbb{N}, \ k\to\infty}
\de^{-\rho k} \frac{\langle x^{x_0,c}(k),{b_0^\star}\rangle^{1-\gamma}}{1-\gamma}= 0. \qedhere
$$
\end{proof}

\begin{theorem}\label{teo:main} Let $x_0\in X_{++}$ {and define} the control
$$
\hat c(t,\theta)\coloneqq (\Phi \hat{x}^{x_0}(t))(\theta)= \left(
\frac{f(\theta)}{\alpha \eta(\theta){b_0^\star}(\theta)}
\right)^{\frac{1}{\gamma}}
\langle x_0, {b_0^\star}\rangle \de^{gt}, \quad t \geq 0, \, \theta \in D,
$$
{where $\hat x^{x_0}$ is the unique mild solution to the closed loop equation~\eqref{cle}}. 
Then, $\hat c \in \mathcal{A}^{{b_0^\star}}_{++}(x_0)$ and it is optimal for problem $(P^{{b_0^\star}})$ starting at $x_0$. Moreover, the value function $V^{b_0^\star}$ of problem $(P^{b_0^\star})$ satisfies $V^{b_0^\star}(x_0)=v(x_0)$, where $v$ is the function defined in~\eqref{solution}.
\end{theorem}
\begin{proof}
Recalling~\eqref{eq:hatx0expl}, $x^{x_0,\hat{c}}=\hat x^{x_0}$ by construction. Hence, by~\eqref{admissibility}, $\hat c\in \mathcal{A}^{{b_0^\star}}_{++}(x_0)$. Moreover, again by construction, $\hat c$ verifies the optimality condition~\eqref{argmaxbis}. So, {thanks to Lemma~\ref{lm:trasv}, the assumptions of Theorem~\ref{verification} are verified, whence we obtain $V^{b_0^\star}(x_0)=v(x_0)$.}
\end{proof}

\subsection{Steady states and stability of solutions}\label{sec:ss}
	\blu{Let us begin by stating an important result concerning the operator $B$, defined in~\eqref{eq:Bdef}, that will motivate what follows.}
\begin{lemma}{\citep[Ch. V, Corollary~3.2]{engelnagel:sgr}}\label{lem:cptsgrspectrum}
{If $B$ generates} an eventually compact $C_0$-semigroup on $X$, then the following properties hold:
\begin{enumerate}
\item The spectrum $\sigma(B)$ is either empty, finite, or countable and consists of poles of the resolvent of finite algebraic multiplicity only.
\item The set $\{\nu \in \sigma(B) \colon \mathfrak{Re} \, \nu \geq r\}$ is finite, for any $r \in \R$.
\end{enumerate}
\end{lemma}

\blu{As a consequence of Lemma~\ref{lem:cptsgrspectrum}, if $\{\de^{tB}\}_{t\geq 0}$ is eventually compact and if $\sigma(B) \neq \emptyset$, {we have that} $\sigma(B) = \{\nu_0, \nu_1, \dots\}$, where $\mathfrak{Re} \, \nu_k \geq \mathfrak{Re} \, \nu_{k+1}$, for any $k \in \mathbb N$ and $\lim_{k \to \infty} \mathfrak{Re} \, \nu_k = -\infty$, provided that $\sigma(B)$ is infinite.
Moreover, by \citep[Ch. IV, Proposition~2.18 (i)]{engelnagel:sgr} we have that $\sigma(B) = \sigma(B^\star)$, and hence Lemma~\ref{lm:gautov} entails that $g \in \sigma(B)$.}

\blu{Notice that $g$ is the exponential growth rate of the map $t\mapsto \langle \hat x^{x_0}(t),b_{0}^{\star}\rangle$, cf.~\eqref{eq:hatx0expl}, where $\hat x^{x_0}$ is the mild solution   to the closed loop equation~\eqref{cle}. This suggests to study convergence and stability of the detrended optimal paths $\{\de^{-gt}\hat x^{x_0}(t)\}_{t \geq 0}$ exploiting the spectral properties of operator $B$. In particular, we aim at using a Perron-Frobenius-type argument. To do so, we need that $g$ is the highest eigenvalue of $B$, i.e., that $g = \nu_0$.
However, we cannot say, \textit{ex ante}, for which $k \in \mathbb{N}$ we have $g=\nu_k$. In Section~\ref{sec:app},  we will see that, under appropriate assumptions, $g=\nu_0$.}

\blu{Motivated by this argument, we are going to provide a stability property of the \emph{detrended optimal paths} $\hat{x}^{x_0}_{\nu_{0}}$ for problem $(P^{b_0^\star})$, where}
$$\hat{x}^{x_0}_{\nu_{0}}(t)\coloneqq \de^{-\nu_{0}t} \hat {x}^{x_0}(t)=\de^{t(B-\nu_{0})}{x_0}, \ \ \ t\geq 0.$$ 
{It is worth noticing that $\hat{x}^{x_0}_{\nu_{0}}$ is the unique (mild) solution} to the \emph{detrended closed-loop equation}
\begin{equation}\label{cleg}
\left\{
\begin{aligned}
&{x}'(t)=(L-\nu_{0}){x}(t)-N\Phi{x}(t), \quad t \geq 0, \\
&{x}(0) = x_0 \in X.
\end{aligned}
\right.
\end{equation}
An element $\bar{x}\in X$ is called a \emph{steady state} for the dynamical system \eqref{cleg} if
\begin{equation}\label{eq:st}
\hat{x}^{\bar{x}}_{\nu_{0}}(t)= \de^{t(B-\nu_{0})}\bar{x}=\bar{x}, \ \ \  \ \ \ \forall t\geq 0,
\end{equation}
From \citep[Ch. IV, Cor. 3.8]{engelnagel:sgr} we deduce that 
\begin{equation}\label{steady} 
\bar{x}\in X \ \mbox{is a steady state for \eqref{cleg}} \ \Longleftrightarrow \ \ \bar x \in \ker (B-\nu_{0}).
\end{equation}
The next theorem provides a stability result under appropriate assumptions.
\begin{theorem}\label{th:stability}
Let $B$ generate an eventually compact $C_0$-semigroup on $X$ and assume that $\nu_{0}$ is a first-order pole of the resolvent, that is,  its algebraic multiplicity is $1$.
Then, for any $x_0 \in X$, there exist $M\geq 1,\,\epsilon>0$ such that
$$| \hat{x}^{x_0}_{{\nu_{0}}}(t) - Px_0|_X \leq M \de^{-\epsilon t}|(1-P)x_0|_X, \quad {\forall t \geq 0,}$$
where $P$ is the spectral projection corresponding to the spectral set $\{\nu_{0}\} \subset \sigma(B)$.
Moreover, $\mathrm{Range}(P)= \ker(B-\nu_{0})$, i.e., $P$ is a projection on $\ker(B-\nu_{0})$, {and $Px_0$ is a steady state for~\eqref{cleg}.}
\end{theorem}
%
%In order to get stability of steady states, we shall impose that $\de^{tB}$ is \emph{eventually compact}. This assumption has important consequences on the spectrum of $B$, as the following lemma shows.
%
%
%
%
%The next theorem provides a stability result under appropriate assumptions.
%\begin{theorem}\label{th:stability}
%Let $B$ generate an eventually compact $C_0$ semigroup on $X$ and assume that:
%\begin{enumerate}[(i)]
%\item
%$\nu_0=g$;
% \item $g$ is dominant; that is, $g > \mathfrak{Re} \, \nu_1$;
%\item $g$ is a first-order pole of the resolvent; that is,  its algebraic multiplicity is $1$.
% \end{enumerate}
% Then, for any $x_0 \in X$,
% there exist $M\geq 1,\,\epsilon>0$ such that
%$$|\hat x_g(t) - Px_0|_X \leq M \de^{-\epsilon t}|(1-P)x_0|_X,$$
%where $P$ is the spectral projection \red{corresponding to the spectral set $\{g\} \subset \sigma(B)$.
%Moreover, $Px \in \ker(B-g)$, for any $x \in X$, i.e., $P$ is a projection on $\ker(B-g)$.}
%\end{theorem}

\begin{proof}
By \citep[Ch. V, Corollary~3.3]{engelnagel:sgr}, there exist constants $\epsilon > 0$ and $M \geq 1$ such that
\begin{equation}\label{est1}
| \de^{t(B-\nu_{0})}-P|_{\mathcal{L}(X)} \leq M \de^{-\epsilon t}, \quad {\forall t \geq 0,}
\end{equation}
where $|\cdot|_{\mathcal{L}(X)}$ denotes the operator norm on the space of linear continuous operators $\mathcal{L}(X)$ and where $P \in \mathcal L(X)$ is the residue of the resolvent at $\nu_0$.

From standard theory (see, e.g., \citep[A-III, Sec. 3]{arendt1986:possgr} or \citep[Ch. IV, Sec. 1]{engelnagel:sgr}), $P$ is the spectral projection corresponding to the spectral set $\{\nu_{0}\} \subset \sigma(B)$. Moreover, being $\nu_{0}$ a simple pole, we have that
$\mbox{Range}(P) = \ker(B-\nu_{0})$ 
(see, e.g., \citep[A-III, Sec. 3, p.~73]{arendt1986:possgr} or \citep[Ch. IV, Sec. 1, p.~247]{engelnagel:sgr}).
{By~\eqref{steady}, this implies that $Px_0$ is a steady state for~\eqref{cleg}. Finally,} using \eqref{steady}--\eqref{est1} and taking into account that $P^2x_0=Px_0$, we have {that, for all $t \geq 0$,}
\begin{align*}
|\hat x_{\nu_{0}}^{x_0}(t) - Px_0|_X 
&= |\de^{t(B-\nu_{0})}x_0 - \de^{t(B-\nu_{0})}Px_0|_X 
\\
&= |\de^{t(B-\nu_{0})}(x_0 - Px_0)-P(x_0-Px_0)|_X \leq M  \de^{-\epsilon t}|(1-P)x_0|_X.\qedhere
\end{align*}
\end{proof}

\begin{remark}\label{rem:explicitP}
For the reader's convenience, we recall that under the assumptions of Theorem~\ref{th:stability} the spectral projection $P$ corresponding to the spectral set $\{\nu_{0}\} \subset \sigma(B)$ is (see \citep[Ch. IV, (1.12)]{engelnagel:sgr}):
\begin{equation}\label{eq:explicitP}
P := -\dfrac{1}{2\pi\mathrm i}\int_\gamma (B-\mu)^{-1} \, \d\mu,
\end{equation}
where $\gamma$ can be taken as the simple curve given by the positively oriented boundary of a disk centered at $\nu_{0}$, with radius small enough so that it does not enclose any other points of $\sigma(B)$.
\end{remark}

\section{\texorpdfstring{Existence of a strictly positive eigenvector of $L^\star$}{Existence of a strictly positive eigenvector of L*}}\label{sec:perronfrobenius}
In this section we state some results guaranteeing that the particularly relevant point~\textit{(\ref{ass:eigenvect})} of Assumption~\ref{ass:explicit} is satisfied. {This fact is tightly linked to some properties of operator $L$, appearing in the state equation~\eqref{state}, that will be discussed.} In the following, recall that we are working under Assumption~\ref{ass:L}.

We start our discussion with the general case where $X$ is any Banach lattice.
Let $\sigma(L)$ be the \emph{spectrum} of the operator $L$ and define the \emph{spectral bound} of $L$:
\begin{equation*}
s_L \coloneqq \sup\{\mathfrak{Re} \lambda \colon \lambda \in \sigma(L)\},
\end{equation*}
with the convention $\sup \emptyset=-\infty$.
Due to Assumption~\ref{ass:L}, $s_L = \sup\{\lambda \in \R \colon \lambda \in \sigma(L)\}$ (see (\citep[Th. 7.4 and Th. 8.7]{clement1987:sgr}) and, if $\sigma(L) \neq \emptyset$, then  $s_L \in \sigma(L)$. The \emph{peripheral} or \emph{boundary spectrum} of $L$ is the subset of $\mathbb{C}$ given by
\begin{equation*}
\sigma_b(L) \coloneqq \{\lambda \in \sigma(L) \colon \mathfrak{Re} \lambda = s_L\}.
\end{equation*}

We introduce the following definitions {(the first one can be generalized, as in \citep[C-III, Def.~3.1]{arendt1986:possgr} or \citep[Prop.~7.6]{clement1987:sgr})}.
\begin{definition}
The {positive $C_0$-semigroup} $\{\de^{tL}\}_{t\geq 0}$ is said to be \emph{irreducible} if for all $f \in X_+\setminus\{0\}$ and  ${\varphi^\star} \in X^\star_+\setminus\{0\}$, there exists $t_0 \geq 0$ such that $\langle \de^{t_0 L} f, {\varphi^\star} \rangle > 0$.
\end{definition}

\begin{definition}
A point $f \in X_+$ is said  a \emph{quasi-interior} point (of $X_+$) if $\langle f, {\varphi^\star} \rangle > 0$ for each $\varphi^\star \in X^\star_+\setminus\{0\}$. \end{definition}

\begin{remark}
If $\mathrm{int}\,{X_+} \neq \emptyset$, the concepts of interior and quasi-interior point coincide; this is the case, for instance, when  $X$ is the space of real-valued continuous functions on a compact set. Instead, in the case $X = \dL^p(D, \mu)$, with $\mu$ a $\sigma$-finite measure and $p \in [1,+\infty)$, we have that $\mathrm{int}\,{X_+} = \emptyset$, while the {set of} quasi-interior points of $X_+$  {coincides with the set of} $\mu$-a.e. strictly positive {functions} (see \citep[p. 238]{arendt1986:possgr}).
\end{remark}

\begin{theorem}\label{th:eig++gen}
Let $\{\de^{tL}\}_{t\geq 0}$ be irreducible and assume that $s_L > -\infty$ is a pole of the resolvent of $L$.
Then,
\begin{enumerate}[(i)]
\item There exists ${\varphi_0^\star} \in D(L^\star) \cap X^\star_{++}$ such that
\begin{equation}\label{eq:eigL*gen}
L^\star {\varphi_0^\star} = s_L {\varphi_0^\star};
\end{equation}
\item $s_L \in \sigma(L)$ has algebraic and geometric multiplicity $1$ and there exist a unique (up to a multiplicative constant) quasi-interior point $f_0 \in D(L) \cap X_+$ such that
\begin{equation}\label{eq:eigLgen}
L f_0 = s_L f_0;
\end{equation}
\item $\sigma_b(L) = s_L + \mathrm{i}\nu\mathbb{Z}$, for some $\nu \geq 0$;
\item $s_L$ is the only eigenvalue of $L$ admitting a positive eigenvector;
\item\label{th:strictposeig} $s_L$ is the only eigenvalue of $L^\star$ admitting a strictly positive eigenvector.
\end{enumerate}
\end{theorem}

\begin{proof}
Items \textit{(i)}--\textit{(iii)} are proved, for instance, in \citep[Th. 8.17]{clement1987:sgr} (see, also, \citep[C-III, Prop. 3.5]{arendt1986:possgr}).

To prove item \textit{(iv)}, suppose that $\lambda \in \mathbb{C}$ is such that $Lf = \lambda f$, where $f \in X_+ \cap D(L)\setminus\{0\}$. Then,
\begin{equation*}
\lambda \langle f, {\varphi_0^\star} \rangle = \langle Lf, {\varphi_0^\star} \rangle = \langle f, L^\star {\varphi_0^\star} \rangle = s_L \langle f, {\varphi_0^\star} \rangle.
\end{equation*}
By \citep[C-III, Prop. 3.5(a)]{arendt1986:possgr}, every positive eigenvector of $L$ must be a quasi-interior point of $X$, hence $\langle f, {\varphi_0^\star} \rangle > 0$; therefore, from the chain of equalities above we get $\lambda = s_L$.

Finally, to prove item \textit{(v)}, suppose that ${\lambda^\star} \in \mathbb{C}$ is such that $L^\star {\varphi^\star} = {\lambda^\star} {\varphi^\star}$, where $\varphi^\star\in X^\star_{++} \cap D(L^\star) \setminus \{0\}$. Then,
\begin{equation*}
\lambda^\star \langle f_0, {\varphi^\star} \rangle = \langle f_0, {\lambda^\star} {\varphi^\star} \rangle = \langle f_0, L^\star {\varphi^\star} \rangle = \langle L f_0, {\varphi^\star} \rangle = s_L \langle f_0, {\varphi^\star} \rangle.
\end{equation*}
Since $f_0 \neq 0$ and $\varphi^\star$ is strictly positive, we have $\langle f_0, {\varphi^\star} \rangle > 0$; therefore the chain of equalities above provides ${\lambda^\star} = s_L$.
\end{proof}

\begin{remark}\label{rem:domeig}
Item (iii) of Theorem~\ref{th:eig++gen} entails that either $\sigma_b(L) = \{s_L\}$ or $\sigma_b(L)$ is an infinite unbounded set. If one is able to exclude the second case or to prove, at least, that the intersection of $\sigma_b(L)$ with the point spectrum of $L$ is the singleton $\{s_L\}$, then, on the one hand, \eqref{eq:eigLgen} implies that $s_L$ is the \emph{dominant eigenvalue} of $L$; on the other hand, \eqref{eq:eigL*gen} implies that point~\textit{(\ref{ass:eigenvect})} of Assumption~\ref{ass:explicit} is satisfied.
\end{remark}

\begin{remark}
The assumption that $s_L > -\infty$, i.e., that $\sigma(L) \neq \emptyset$, in Theorem~\ref{th:eig++gen} is essential, since there are examples of positive irreducible $C_0$-semigroups on Banach lattices such that $\sigma(L) = \emptyset$ (see, e.g. \citep[C-III, Example 3.6]{arendt1986:possgr}). Some conditions that, together with irreducibility, imply that this is not the case are stated in \citep[C-III, Th. 3.7]{arendt1986:possgr}.
\end{remark}

\begin{remark}
Theorem~\ref{th:eig++gen} does not guarantee \emph{uniqueness} (up to normalization) of a strictly positive eigenvector of $L^\star$ associated to $s_L$, since we do not know if $s_L$ is a \emph{geometrically simple} eigenvalue of $L^\star$.
\end{remark}

We conclude our discussion, by looking at the specific case where $X = \calc(D)$, the space of real-valued continuous functions on a set $D$, that we suppose to be a compact separable topological space. {In this case,} $X^\star$ is the space of bounded regular Borel measures on $D$.
{The following result is particularly useful to check that point~\textit{(\ref{hyp:L})} of Assumption~\ref{ass:L} is satisfied.
\begin{proposition}\label{prop:Lstrictpos}
If $X = \calc(D)$, with $D$ a compact separable topological space, and if $\{\de^{tL}\}_{t\geq 0}$ is a positive $C_0$-semigroup, then $\{\de^{tL}\}_{t\geq 0}$ preserves strict positivity, i.e., $\de^{{tL}}(X_{++})\subseteq X_{{++}}$ for all $t\geq 0$. Moreover, the spectrum of $L$ is not empty and $-\infty < s_L \in \sigma(L)$.
\end{proposition}
\begin{proof}
For the first part of the statement, see \citep[B-II, Cor. 1.17]{arendt1986:possgr}. A proof of the last assertion is given in \citep[B-III, Th. 1.1]{arendt1986:possgr}.
\end{proof}
The following theorem provides a sufficient condition to ensure that point~\textit{(\ref{ass:eigenvect})} of Assumption~\ref{ass:explicit} is verified.}
\begin{theorem}\label{th:eig++cont}
{If $X = \calc(D)$, with $D$ a compact separable topological space, and if $\{\de^{tL}\}_{t\geq 0}$ is a positive $C_0$-semigroup,} the spectral and the growth bound of $L$ coincide and there exists a positive probability measure $0 \neq {\varphi^\star} \in D(L^\star) \cap X^\star_+$ such that
\begin{equation}\label{eq:eigL*cont}
L^\star {\varphi^\star} = s_L {\varphi^\star}.
\end{equation}
Moreover, if $\{\de^{tL}\}_{t\geq 0}$ is irreducible and if $s_L$ is a pole of the resolvent of $L$, then $\varphi^\star \in D(L^\star) \cap X^\star_{++}$, i.e., $\varphi^\star$ is strictly positive.
\end{theorem}

\begin{proof}
The proof of the first statement of the theorem can be found in \citep[B-III, Th. 1.6]{arendt1986:possgr} (see also the discussion that follows). The last statement follows from Theorem~\ref{th:eig++gen}, item (i) (see, also, \citep[B-III, Prop. 3.5]{arendt1986:possgr}).
\end{proof}

\section{Application to an {economic growth problem with space heterogeneity}}\label{sec:app}
We apply the results of the previous sections to a family of economic growth problems where the state variable, representing capital,  is \emph{space heterogeneous}, in the sense that it depends not only on time, but also on the space location.
Problems of this kind have been recently studied in the economic literature (see \cite{boucekkine:spatialAK, BFFGjoeg, BFFGpafa} for more details), embedding them in an infinite dimensional state space of $\dL^2$ type. As mentioned in the Introduction, this did not allow to treat satisfactorily the presence of positivity constraints. Here, instead, we embed the problem in the space of continuous functions. 

\blu{To introduce this family of economic growth problems, we specialize the abstract setting of Section~\ref{sec:model} with the following assumption, that will be in force throughout this section.}
{\begin{assumption}\label{ass:econpb}
\begin{enumerate}[(a)]
\item[]
\item
$
D=\mathbf{S}^1\coloneqq\{\xi\in \R^2 \colon |\xi|_{\R^2}=1\}\cong 2\pi\R/\mathbb{Z}$. The space $\mathbf{S}^1$ is topologically equivalent to $[0,2\pi]\subset\R$, whenever the endpoints of the latter interval are identified. Similarly, functions on $\mathbf{S}^1$ are identified with $2\pi$-perodic functions on $\R$.
\item $\mu$ is the Hausdorff measure on $\mathbf{S}^1$. Through the identification $\mathbf{S}^1\cong 2\pi\R/\mathbb{Z}$, we will understand $\mu$ as the Lebesgue measure on $[0,2\pi]$. \blu{Accordingly, we will write $\d\theta$ instead of $\mu(\d\theta)$, when integrating functions with respect to $\mu$.}
\item \blu{$X=\calc(\mathbf{S}^1)$, endowed with the sup-norm --- i.e., $|x|_X=|x|_\infty \coloneqq \sup_{\theta \in D} |x(\theta)|$ --- and with the pointwise ordering --- i.e., $x \leq_X y \, \Longleftrightarrow \, x(\theta) \leq y(\theta)$ for all $\theta \in \mathbf{S}^1$.}
\item  \blu{The operators $L$ and $N$, that define the state equation~\eqref{state}, and the function $u$ of Assumption~\ref{ass:u} are explicitly given by
\begin{align}
[Lz](\theta)&=\sigma \frac{\d^2}{\d\theta^2}z(\theta) + A(\theta)z(\theta), & &z \in X, \, \theta \in D, \label{eq:Lexpl}
\end{align}
with domain  $D(L)=\calc^2(\mathbf{S}^1)$,
\begin{align}
[Nz](\theta)&=\eta(\theta)z(\theta), & &z \in X, \, \theta \in D, \label{eq:Nexpl2}
\\
u(\theta,\xi)&=\displaystyle{\frac{\xi^{1-\gamma}}{1-\gamma} \eta(\theta)^q}, & &\theta \in D, \, \xi \in \R_+, \label{eq:uexpl2}
\end{align}
where $A, \eta \in \calc(\mathbf S^1; (0,+\infty))$ are given functions and $\sigma > 0$, $q \geq 0$, $\gamma\in (0,1)\cup(1,+\infty)$ are fixed constants.}
\end{enumerate}
\end{assumption}
\blu{\begin{remark}\label{rem:Nuexpl}
Between the two possible choices of $X$ in Assumption~\ref{ass:explicit}, we are considering the set of real-valued continuous functions on a compact set. Notice that the operator $N$ has the same expression given in~\eqref{eq:Nexpl}, while $u$ is of the same type as in~\eqref{eq:uexpl}, with $f = \eta^q$.
As already observed in point~\textit{(\ref{rem:etauhyp})} of Remark~\ref{rem:expl}, this implies that point~\textit{(\ref{hyp:N})} of Assumption~\ref{ass:L} and Assumption~\ref{ass:u} are verified. 
In the following, we will prove that point~\textit{(\ref{hyp:L})} of Assumption~\ref{ass:L} and points~\textit{(\ref{ass:eigenvect})}--\textit{(\ref{ass:C})} of Assumption~\ref{ass:explicit} are satisfied, too.
\end{remark}
}

\blu{
This specific setting corresponds to economic growth models, where the output of the economy is described by an AK production function, a $1$-dimensional geography is explicitly taken into account through a variable $\theta \in \mathbf{S}^1$, and capital diffusion over the geographical space is considered. Let us briefly summarize these problems (we refer to \citep{boucekkine:spatialAK, BFFGjoeg, BFFGpafa} for further details).
}

\blu{
The state variable is the capital level $K(t,\theta)$ at time $t\geq 0$ and at location $\theta\in \mathbf{S}^{1}$. Its evolution depends on the consumption policy $C(t,\theta)$, that has to be determined by a social planner at each time $t\geq 0$ and at each location $\theta\in \mathbf{S}^{1}$ with the aim of maximizing a reward functional. More precisely, for any initial capital endowment $K_0 \in \calc(\mathbf{S}^1;(0,+\infty))$, the time-space evolution of the capital level obeys the following PDE
\begin{equation}\label{statebis}
\begin{cases}\displaystyle{
\frac{\partial K}{\partial t}(t,\theta)=\sigma \frac{\partial^2K}{\partial \theta^2}(t,\theta)+A(\theta)K(t,\theta)-\eta(\theta)C(t,\theta),\  \ \ \  (t,\theta)\in \R_+\times \mathbf{S}^1,}\\\\
K(0,\theta)=K_0(\theta), \ \ \ \ \  \ \theta\in \mathbf{S}^1.
\end{cases}
\end{equation}
The data appearing in this PDE are the functions $\eta, A$, where $\eta(\theta)$ represents the population density at location $\theta\in \mathbf{S}^{1}$, while $A(\theta)$ represents the technological level of the economy at location $\theta\in\mathbf{S}^{1}$, and a diffusion coefficient $\sigma$  measuring how fast the capital spreads over the geographical space.
The optimization problem that the social planner aims to solve is
\begin{equation}\label{eq:econP}
V_{++}(K_0)\coloneqq \sup_{\mathcal{A}_{++}(K_0)}
\int_0^\infty \de^{-\rho t}\left( \int_{\mathbf{S}^1}  \frac{C(t,\theta)^{1-\gamma}}{1-\gamma}\eta(\theta)^q \d\theta \right)\d t,
\end{equation}
where $\rho > 0$ is a fixed discount factor, the parameters $q \geq 0$ and $\gamma \in (0,1) \cup (1,+\infty)$ describe the preferences of the social planner, and
\begin{equation}\label{eq:A++}
\mathcal{A}_{++}(K_0)\coloneqq\big\{C\in \dL^1_{loc}([0,+\infty);\calc(\mathbf{S}^1;\R^+)) \colon K(t,\theta)> 0 \text{ for all } (t,\theta)\in \R_+\times \mathbf{S}^1\big\}.
\end{equation}
It is clear that under Assumption~\ref{ass:econpb} this problem can be reformulated in the abstract setting introduced in Section~\ref{sec:model}. More specifically, with the identifications $x_0 \coloneqq K_{0}$, $x(t) \coloneqq K(t,\cdot)$, $c(t) \coloneqq C(t,\cdot)$, PDE~\eqref{statebis} corresponds to the state equation~\eqref{state}, the optimization problem~\eqref{eq:econP} coincides with problem~\eqref{eq:pbP}, and the class of admissible consumption plans~\eqref{eq:A++} is precisely the class of admissible controls~\eqref{eq:A++def}.
}
\blu{As in Section~\ref{sec:hjb}, we can relax the state constraint~\eqref{eq:A++} and consider the optimization problem
\begin{equation}\label{eq:econPb0}
V^{b_0^\star}(K_0)\coloneqq \sup_{\mathcal{A}^{b_0^\star}_{++}(K_0)}
\int_0^\infty \de^{-\rho t}\left( \int_{\mathbf{S}^1}  \frac{C(t,\theta)^{1-\gamma}}{1-\gamma}\eta(\theta)^q \d\theta \right)\d t,
\end{equation}
where $b_0^\star$ is the eigenvector of point~\textit{(\ref{ass:eigenvect})} of Assumption~\ref{ass:explicit} and
$$
\mathcal{A}^{b_0^\star}_{++}(K_0)\coloneqq\big\{C\in \dL^1_{loc}(\R^+;\calc(\mathbf{S}^1;\R^+)) \colon \langle K(t,\cdot),b_0^\star\rangle > 0 \mbox{ for all } t\in \R_+\times \mathbf{S}^1\big\}.
$$
Exploiting the results proved in the previous sections, we are able to find an explicit solution to problem~\eqref{eq:econPb0} and a feedback optimal control. Finally, studying the corresponding closed loop equation, we will prove optimality of this control for problem~\eqref{eq:econP}, at least for a suitable set of initial conditions $K_0$.}

Let us consider the Hilbert space $\dL^2(\mathbf{S}^1)$ of square integrable functions $x:\mathbf{S}^1\to \R$ endowed with inner product and induced norm
\begin{equation*}
(x, y)_{\dL^2(\mathbf{S}^1)} \coloneqq \int_{\mathbf{S}^1} x(\theta) y(\theta) \, \d\theta, \ \ \  x,y \in \dL^2(\mathbf{S}^1);
\quad
|x|_{\dL^2(\mathbf{S}^1)} \coloneqq \biggl(\int_{\mathbf{S}^1} x(\theta)^2 \, \d\theta\biggr)^{\frac 12}, \ \ \ x \in \dL^2(\mathbf{S}^1).
\end{equation*}

Each element $f\in \dL^2(\mathbf{S}^1)$ can be seen as an element of  $\calc(\mathbf{S}^1)^\star$ through
$\langle x,f\rangle:=(x,f)_{\dL^2(\mathbf{S}^1)}$
and we have the continuous embeddings $\calc(\mathbf{S}^1)\hookrightarrow \dL^2(\mathbf{S}^1)\hookrightarrow \calc(\mathbf{S}^1)^\star.$
\blu{The next proposition ensures that also point~\textit{(\ref{hyp:L})} of Assumption~\ref{ass:L} is satisfied (for point~\textit{(\ref{hyp:N})} of the same assumption, cf. Remark~\ref{rem:Nuexpl}).}
\begin{proposition}
By  \citep[B-III, Examples 2.14(a) and 3.4(e)]{arendt1986:possgr}) and \citep[B-II, Cor. 1.17]{arendt1986:possgr}, we have the following results.
\begin{enumerate}[(i)]
\item The operator $L \colon D(L)=\calc^2(\mathbf{S}^1)\subset \calc(\mathbf{S}^1)\to \calc(\mathbf{S}^1)$, defined in~\eqref{eq:Lexpl}, is a closed operator generating a positive irreducible $C_0$-semigroup in the space $\calc(\mathbf{S}^1)$.
\item  Assumption~\ref{ass:L} is satisfied.
\end{enumerate}
\end{proposition}

\blu{Our first aim is to use Theorem~\ref{teo:main} to provide an explicit solution to problem~\eqref{eq:econPb0}. To do so, we need, first, to ensure that Assumption~\ref{ass:explicit} is verified.}

By standard results from Sturm-Liouville theory with periodic boundary conditions (see, e.g., \citep[Ch.\,7, Th.\,2.1 and Th.\,4.1; Ch.\,8,\,Th.3.1]{coddington:ode} --- cf. also \citep{brown:periodic}), we have the following results.
\begin{proposition}\label{prop:Lspectrum}
\begin{enumerate}[(i)]
\item[]
\item\label{prop:Leig}
The eigenvalues of $L$ are real and form a countable discrete set $\{\lambda_k\}_{k \in \mathbb N}$ such that $\lambda_0 > \lambda_1 > \cdots$, and $\lim_{k \to \infty} \lambda_k = -\infty$.
\item The eigenvalues $\lambda_k$, for $k\geq 1$, have geometric multiplicty $2$, while the highest eigenvalue, $\lambda_0$,  is simple.
The corresponding (normalized) eigenvectors $\{b_k\}_{k \in \mathbb N}$ form an orthonormal basis of $\dL^2(\mathbf{S}^1)$.
\item
 Each $b_k$ has exactly $k$ zeros in $\mathbf{S}^1$; without loss of generality, we take $b_0 > 0$;  this is, up to positive a multiplicative constant, the unique strictly positive eigenvector of $L$.
\item For all $x \in \calc(\mathbf{S}^1)$,
\begin{equation}\label{lem:unifconv}x = \sum_{k=0}^\infty (x, b_k)_{\dL^2(\mathbf{S}^1)} \, b_k :=\lim_{n\to\infty}\sum_{k=0}^n (x, b_k)_{\dL^2(\mathbf{S}^1)} \, b_k, \quad \mbox{in } |\cdot|_\infty.
\end{equation}
\item For all $\mu\in\mathbb{C}\setminus \{\lambda_k\}_{k\in\mathbb{N}}$, the operator $L-\mu:\calc^2(\mathbf{S}^1)\to \calc(\mathbf{S}^1)$ is invertible, the inverse $(L-\mu)^{-1}:\calc(\mathbf{S}^1)\to \calc^2(\mathbf{S}^1)$ is bounded and admits the representation
\begin{equation}\label{resolvent}
(L-\mu)^{-1}x=\sum_{k=0}^\infty \dfrac{(b_k,x)_{\dL^2(\mathbf{S}^1)}}{\lambda_k - \mu} b_k:=\lim_{n\to\infty}\sum_{k=0}^n \dfrac{(b_k,x)_{\dL^2(\mathbf{S}^1)}}{\lambda_k - \mu} b_k, \quad \mbox{in } |\cdot|_\infty.
\end{equation}
In particular the resolvent set of $L$ is $\varrho(L)=\mathbb{C}\setminus\{\lambda_k\}_{k\geq 0}$.
\end{enumerate}
\end{proposition}

\begin{remark}
The last assertion of Proposition~\ref{prop:Lspectrum} entails that the spectrum of operator $L$ is $\sigma(L) = \{\lambda_k\}_{k\geq 0}$. Moreover, using \eqref{lem:unifconv} and the fact that $L$ is closed, we have that
\begin{equation}\label{eq:Lsumswap}
Lx = \sum_{k \in \mathbb N} (x, b_k)_{\dL^2(\mathbf{S}^1)} \lambda_k b_k, \quad  \mbox{in }  |\cdot|_\infty, \quad x \in \calc^2(\mathbf{S}^1).
\end{equation}
\end{remark}

From the facts above, we have that  $s_L = \lambda_0$ is the dominant eigenvalue of the operator $L$ and Theorem~\ref{th:eig++cont} guarantees that there exists a strictly positive eigenvector ${b_0^\star} \in D(L^\star) \cap \calc(\mathbf{S}^1)^\star_{++}$ of $L^\star$ associated to the eigenvalue $\lambda^\star_0=\lambda_0$. Moreover, thanks to point~\textit{(\ref{th:strictposeig})} of Theorem~\ref{th:eig++gen}, $\lambda_0$ is the only eigenvalue of $L^\star$ admitting a strictly positive eigenvector. Hence, point~\textit{(\ref{ass:eigenvect})} of Assumption~\ref{ass:explicit} is satisfied.
The next proposition characterizes $b_0^\star$. 

\begin{proposition} \blu{Every strictly positive eigenvector of $L^{\star}$ associated to $\lambda_0$ is in the form
$b_0^\star=\kappa b_0$ for some $\kappa>0$.}
\end{proposition}
\begin{proof}
Take $x\in \calc(\mathbf{S}^1)$. Using \eqref{eq:Lsumswap} and by uniform convergence, we get
\begin{equation*}
\langle Lx,b_0\rangle = \sum_{k=0}^\infty (x,b_k)_{\dL^2(\mathbf{S}^1)}\lambda_k\langle b_k,b_0\rangle = \lambda_0(x,b_0)_{\dL^2(\mathbf{S}^1)} |b_0|^2_{\dL^2(\mathbf{S}^1)}=\langle x,\lambda_0 b_0\rangle.
\end{equation*}
Since $x\in \calc(\mathbf{S}^1)$ is arbitrary, we obtain that $b_0\in D(L^\star)$ and that $L^\star b_0=\lambda_0 b_0$.
\end{proof}

\blu{In the following, we set $\kappa:=1$, i.e., we take $b_{0}^{\star} \coloneqq b_{0}$. {The result above,  together with continuity and positivity of the function $\eta$ (recall also that $f = \eta^q$ in this section), shows that point~\textit{(\ref{ass:C})} of Assumption~\ref{ass:explicit} is satisfied. Moreover, thanks to compactness of $D = \mathbf{S}^1$, we get that also point~\textit{(\ref{ass:gamma>1})} of Assumption~\ref{ass:explicit} is verified.
Finally, assuming that point~\textit{(\ref{ass:lambda0})} of Assumption~\ref{ass:explicit} holds, i.e., that
\begin{equation}\label{ass:rho}
\rho>{\lambda_0}(1-\gamma),
\end{equation}
all the hypotheses of Theorem~\ref{teo:main} are verified and hence, by making explicit the latter for the economic problem introduced above, we get} the following.}
\blu{\begin{theorem}\label{th:econexample}
We have the following facts.
\begin{enumerate}[(i)]
\item The value function of problem~\eqref{eq:econPb0} is
\begin{equation}\label{VVFF}
V^{b_0^\star}(K_0)= \frac{\alpha}{1-\gamma} \left(\int_{\mathbf{S}^1} K_0(\theta) \, {b_0}(\theta) \, \d\theta)\right)^{1-\gamma},
\end{equation}
where
\begin{equation}\label{eq:alphabis}
\alpha\coloneqq \biggl[\frac{\gamma}{\rho - \lambda_0(1-\gamma)} \int_{\mathbf{S}^1} \eta(\theta)^{\frac{q+\gamma-1}{\gamma}} b_0(\theta)^{\frac{\gamma-1}{\gamma}} \, \d\theta \biggr]^{\gamma};
\end{equation}
\item The control
\begin{equation}\label{cc}
\hat{C}(t,\theta):
=\left(\int_{\mathbf{S}^1} {K}_0(\xi) \, {b_0}(\xi) \, \d\xi\right) \bigl(\alpha {b_0}(\theta)\bigr)^{-\frac 1\gamma} \eta(\theta)^{\frac{q-1}{\gamma}}\de^{gt},
\end{equation}
where $g\coloneqq\dfrac{\lambda_0-\rho}{\gamma}$, 
is optimal for problem~\eqref{eq:econPb0} and the corresponding optimal capital $\hat{K}$ solves, for $(t,\theta)\in \R_+\times \mathbf{S}^1$, the linear integro-PDE
\begin{equation}\label{clebis}
\begin{cases}
\displaystyle{\frac{\partial K}{\partial t}(t,\theta)=\sigma \frac{\partial^2K}{\partial \theta^2}(t,\theta)\!+\!A(\theta)K(t,\theta)\!-\!\bigl(\alpha {b_0}(\theta)\bigr)^{-\frac 1\gamma}\!\! \left(\int_{\mathbf{S}^1} \!\!\!K(t,\xi) \, {b_0}(\xi) \, \d\xi\right)\!\eta(\theta)^{\frac{q+\gamma-1}{\gamma}}},\medskip\\
K(0,\theta)=K_0(\theta).
\end{cases}
\end{equation}
\end{enumerate}
\end{theorem}
}
\blu{\begin{remark}
The value $\displaystyle{g\coloneqq\frac{\lambda_0-\rho}{\gamma}}$  is the \emph{optimal growth rate of the economy}.
\end{remark}}

\blu{Our second aim is to use the stability result provided by Theorem~\ref{th:stability} to show that the optimal control $\hat{C}$ for problem~\eqref{eq:econPb0}, defined in~\eqref{cc}, is optimal also for problem~\eqref{eq:econP}, at least for a suitable set of initial conditions $K_0$.
To achieve this, we need to analyze the semigroup generated by the operator $B = L - N\Phi$. {By~\eqref{clebis}, its explicit expression is, for all $x \in \calc^2(\mathbf{S}^1)$,}
\begin{equation}\label{eq:opB}
(Bx)(\theta) = \sigma \frac{\d^2 x}{\d \theta^2}(\theta) + A(\theta) x (\theta) -\bigl(\alpha {b_0}(\theta)\bigr)^{-\frac 1\gamma} \left(\int_{\mathbf{S}^1} x(\xi) \, {b_0}(\xi) \, \d\xi\right)\eta(\theta)^{\frac{q+\gamma-1}{\gamma}}, \, \theta \in \mathbf{S}^1.
\end{equation}}
\blu{
\begin{lemma}\label{lem:cptsgr}
$\{\de^{tB}\}_{t\geq 0}$ is (immediately) compact.
\end{lemma}
}
\begin{proof}
\blu{We start by showing that the semigroup generated by the operator $L$ on $X$ is (immediately) compact. In fact, $L$ generates an analytic semigroup (see, e.g., \citep[Ch.\,VI,\,Sec. 4]{engelnagel:sgr}),  and the resolvent of $L$ is non-empty; indeed, it is the complement set of the spectrum of $L$, which is a purely point spectrum $\sigma(L) = \{\lambda_0, \lambda_1, \dots\}$ by point~\textit{(\ref{prop:Leig})} of Proposition~\ref{prop:Lspectrum}.}

\blu{
From \citep[Ch. II, (4.26)]{engelnagel:sgr} we deduce that $\{\de^{tL}\}_{t\geq 0}$ is (immediately) norm continuous. Then, since the canonical injection $\iota \colon (\calc^2(\mathbf{S}^1), |\cdot|_{\calc^2(\mathbf{S}^1)}) \to (\calc(\mathbf{S}^1), |\cdot|_\infty)$ is compact by the Ascoli-Arzel\`a Theorem, we can apply \citep[Ch. II, Prop. 4.25]{engelnagel:sgr} and deduce that $L$ has compact resolvent. Finally, we learn from \citep[Ch. II, Th. 4.29]{engelnagel:sgr} that $\{\de^{tL}\}_{t\geq 0}$ is (immediately) compact.}

\blu{Next, notice that $B$ is an additive perturbation of $L$, obtained by subtracting to it the operator $N\Phi$, which is bounded. Therefore, from \citep[Ch. III, Prop. 1.16(i)]{engelnagel:sgr}, we deduce that $\{\de^{tB}\}_{t\geq 0}$ is (immediately) compact.}
\end{proof}

Next, we are going to prove that, under suitable assumptions, $g$ is a simple eigenvalue of $B$ and to characterize its associated eigenvector. To ease notations, let us define the constant
\begin{equation}\label{eq:alpha0}
\alpha_0 \coloneqq \alpha^{\frac{1}{1-\gamma}} = \biggl[\frac{\gamma}{\rho - \lambda_0(1-\gamma)} \int_{\mathbf{S}^1} \eta(\theta)^{\frac{q+\gamma-1}{\gamma}} b_0(\theta)^{\frac{\gamma-1}{\gamma}} \, \d\theta \biggr]^{\frac{\gamma}{1-\gamma}},
\end{equation}
and the function
\begin{equation}\label{eq:betastar0}
\beta(\theta) \coloneqq \alpha_0 b_0(\theta), \quad \theta \in \mathbf{S}^1.
\end{equation}
 Using \eqref{eq:alpha0} and \eqref{eq:betastar0}, we can rewrite the expression for operator $N\Phi$, {appearing in~\eqref{eq:opB} and defined in~\eqref{def:Nphi}, as follows}
\begin{equation}\label{eq:NPhi}
(N\Phi) x = \bigl(\alpha {b_0}\bigr)^{-\frac 1\gamma} \left(\int_{\mathbf{S}^1} x(\xi) \, {b_0}(\xi) \, \d\xi\right)\eta^{\frac{q+\gamma-1}{\gamma}} = {\beta}^{-\frac{1}{\gamma}} \eta^{\frac{q+\gamma-1}{\gamma}} \langle x, \beta \rangle, \quad x \in \calc(\mathbf{S}^1).
\end{equation}
\medskip
Notice that, under our assumptions, $\beta^{-\frac{1}{\gamma}} \eta^{\frac{q+\gamma-1}{\gamma}}\in\calc(\mathbf{S}^1)$, hence it admits a Fourier series expansion with uniform convergence with respect to $\{b_k\}_{k \in \mathbb N}$, with coefficients:
\begin{align}
&\beta_{0}:=(b_0,{\beta}^{-\frac{1}{\gamma}} \eta^{\frac{q+\gamma-1}{\gamma}})_{\dL^2(\mathbf{S}^1)}
= \frac{\mu_0}{\alpha_0},  & &\mbox{where} \ \mu_0 \coloneqq (\beta,{\beta}^{-\frac{1}{\gamma}} \eta^{\frac{q+\gamma-1}{\gamma}})_{\dL^2(\mathbf{S}^1)},\label{eq:beta0}
\\
&\beta_{k}:=(b_k,{\beta}^{-\frac{1}{\gamma}} \eta^{\frac{q+\gamma-1}{\gamma}})_{\dL^2(\mathbf{S}^1)}, & &k \geq 1. \label{eq:betak}
\end{align}
Using \eqref{eq:alpha0}, it is immediate to see that $\mu_{0}=\lambda_{0}-g$. 
Hence
$$
{\beta}^{-\frac{1}{\gamma}} \eta^{\frac{q+\gamma-1}{\gamma}}=\dfrac{\lambda_0-g}{\alpha_0} b_0+\sum_{k=1}^\infty \beta_k b_k,
$$
with convergence in $\calc(\mathbf{S}^1)$.
Consider the formal series
\begin{equation}\label{eq:e0}
 \dfrac{b_0}{\alpha_0} + \sum_{k=1}^{\infty} \dfrac{\beta_k}{\lambda_k - g} b_k,
\end{equation}
%where $\beta_k$, $k \geq 1$, are the coefficients defined in \eqref{eq:betak}.
\begin{proposition}\label{ass:w}
Let \eqref{ass:rho} hold. The series \eqref{eq:e0} converges in $\calc(\mathbf{S}^1)$ and defines a function of  $\calc^2(\mathbf{S}^1)$.
\end{proposition}
	
\begin{proof}
By \eqref{ass:rho}, $g$ is not an eigenvalue of $L$; moreover, ${\beta}^{-\frac{1}{\gamma}} \eta^{\frac{q+\gamma-1}{\gamma}}\in\calc(\mathbf{S}^1)$. Using \eqref{resolvent}, \eqref{eq:beta0}, \eqref{eq:betak}, and recalling that $\mu_0 = \lambda_0 - g$, we get
\begin{equation}\label{eq:w}
(L-g)^{-1} \left({\beta}^{-\frac{1}{\gamma}} \eta^{\frac{q+\gamma-1}{\gamma}}\right) = \sum_{k=0}^\infty \dfrac{(b_k,{\beta}^{-\frac{1}{\gamma}} \eta^{\frac{q+\gamma-1}{\gamma}})_{\dL^2(\mathbf{S}^1)}}{\lambda_k - g} \, b_k = \dfrac{b_0}{\alpha_0} + \sum_{k=1}^{\infty} \dfrac{\beta_k}{\lambda_k - g} b_k. \qedhere
\end{equation}
%Therefore, the series \eqref{eq:e0} converges in $\calc(\mathbf{S}^1)$ and it is equal to $\calg_g \left({\beta}^{-\frac{1}{\gamma}} \eta^{\frac{q+\gamma-1}{\gamma}}\right) \in \calc^2(\mathbf{S}^1)$.
\end{proof}

Given the proposition above, we set
\begin{equation}\label{eq:e00}
w \coloneqq \dfrac{b_0}{\alpha_0} + \sum_{k=1}^{\infty} \dfrac{\beta_k}{\lambda_k - g} b_k\in \calc^2(\mathbf{S}^1).
\end{equation}

We are ready to state the following result on the spectrum of operator $B$.
\begin{proposition}\label{prop:spectrumB}
Let \eqref{ass:rho} hold and assume that $g > \lambda_1$.
Then, we have the following facts.
\begin{itemize}
\item[(i)] The spectrum of the operator $B$ is given by $\sigma(B) = \{g\}\cup \{\lambda_k\}_{k \in \mathbb N \setminus \{0\}}$, and hence $g$ is the dominant eigenvalue of $B$;
\item[(ii)] $g$ is a simple eigenvalue of $B$;
\item[(iii)] $\ker (B-g)=\mathrm{Span}\left\{{w}\right\}$, \blu{where $w$ is given in~\eqref{eq:e00}.}
\item[(iv)] The spectral projection $P$ corresponding to the spectral set $\{g\} \subset \sigma(B)$ is the bounded and linear operator:
\begin{equation}\label{eq:opPgen}
Px = \langle x, \beta \rangle w, \quad x \in \calc(\mathbf{S}^1).
\end{equation}
\end{itemize}
\end{proposition}

\begin{proof}
\noindent(i)
We start showing that $g \in \sigma(B)$. By \eqref{ass:rho}, $g$ is not an eigenvalue of $L$. Thus,  \citep[Ch.\,IV, Prop.\,4.2]{engelnagel:sgr} entails that $g$ is an eigenvalue of $B = L - N\Phi$ if and only if $1$ is an eigenvalue of $(L-g)^{-1} N\Phi$, i.e., if and only if the equation
\begin{equation}\label{eq:eig1}
(L-g)^{-1} N\Phi x =x,
\end{equation}
admits a nontrivial solution $x \in \calc(\mathbf{S}^1)$. By \eqref{eq:NPhi} and \eqref{eq:w}, $(L-g)^{-1} N\Phi x = \langle x, \beta \rangle w$, for all $x\in \calc(\mathbf{S}^1)$, and hence \eqref{eq:eig1} becomes $\langle x, \beta \rangle w =x$.
By \eqref{eq:e00} we have $\langle w, \beta \rangle = (w, \beta)_{\dL^2(\mathbf{S}^1)} = (b_0,b_0)_{\dL^2(\mathbf{S}^1)}=1$. Therefore, we obtain that $w$ is a nontrivial solution to \eqref{eq:eig1} and, thus, $g$ is an eigenvalue of $B$.

Next, we prove that $\lambda_0 \notin \sigma(B)$. 
Suppose, by contradiction, that $\lambda_0$ is an eigenvalue of $B$ and that $x$ is an associated eigenfunction.
Let  $x_k \coloneqq (x, b_k)_{\dL^2(\mathbf{S}^1)}$. By \eqref{eq:Lsumswap}, recalling \eqref{eq:NPhi} and that ${\beta}^{-\frac{1}{\gamma}} \eta^{\frac{q+\gamma-1}{\gamma}}$ admits a Fourier series expansion whose coefficients are given in \eqref{eq:beta0}--\eqref{eq:betak}, we get
\begin{equation*}
0 = B x - \lambda_0 x = \sum_{k=0}^\infty x_k \lambda_k b_k -\left(\dfrac{\mu_0}{\alpha_0} b_0 + \sum_{k=1}^\infty \beta_k b_k \right)x_0 \alpha_0 - \sum_{k=0}^\infty \lambda_0 x_k b_k.
\end{equation*}
This equation is satisfied if and only if
\begin{equation*}
\left\{
\begin{aligned}
&\lambda_0 x_0 - \mu_0 x_0 = \lambda_0 x_0, \\
&\lambda_k x_k - x_0 \alpha_0 \beta_k = \lambda_0 x_k , \qquad k\in\mathbb{N}\setminus\{0\}.
\end{aligned}
\right.
\end{equation*}
Since $\mu_0 = \lambda_0 - g$, which entails $\mu_0 \neq 0$ by \eqref{ass:rho}, the first equation implies $x_0 = 0$. Substituting  into the second equation, we get $x_k(\lambda_k -\lambda_0) = 0$, for all $k \geq 1$. Since $x\neq 0$, there must be an index $k \geq 1$ such that $\lambda_k = \lambda_0$, a contradiction.

Finally, to show that $\lambda_k$, $k \geq 1$, are the only other eigenvalues of $B$, it is enough to prove that $\mu \in \mathbb C \setminus\{g, \lambda_0\}$ is an eigenvalue of $B$ if and only if $\mu$ is an eigenvalue of $L$. This is readily verified with computations similar to those above.\smallskip

\noindent(ii) We start by recalling that, since $\{\de^{tB}\}_{t\geq 0}$ is a compact semigroup (Lemma~\ref{lem:cptsgr}), we know from Lemma~\ref{lem:cptsgrspectrum} that $g$ is a pole of the resolvent of $B$ of finite algebraic multiplicity.
Let $h \in \calc(\mathbf{S}^1)$, $\mu \in \mathbb C \setminus \sigma(B)$, set $h_k \coloneqq (h, b_k)_{\dL^2(\mathbf{S}^1)}$ for $k \in \mathbb N$.
We claim that
\begin{equation}\label{eq:resolvB}
(B-\mu)^{-1} h = (L-\mu)^{-1}\left(h + \dfrac{h_0 \alpha_0}{g-\mu} {\beta}^{-\frac{1}{\gamma}} \eta^{\frac{q+\gamma-1}{\gamma}}\right).
\end{equation}
First, notice that, since $\mu\neq g$, we have $h + \dfrac{h_0 \alpha_0}{g-\mu} {\beta}^{-\frac{1}{\gamma}} \eta^{\frac{q+\gamma-1}{\gamma}}\in\calc(\mathbf{S}^1)$, so the right hand side above is well defined and belongs to $\calc^2(\mathbf{S}^1)$.
Next, recalling \eqref{resolvent} and given that $N\Phi$ is a bounded linear operator, we get:
{\allowdisplaybreaks
\begin{align}
\notag 
&\mathop{\phantom{=}} N\Phi (L-\mu)^{-1} \left(h + \dfrac{h_0 \alpha_0}{g-\mu} {\beta}^{-\frac{1}{\gamma}} \eta^{\frac{q+\gamma-1}{\gamma}}\right) \\
\label{eq:2} 
&= {\beta}^{-\frac{1}{\gamma}} \eta^{\frac{q+\gamma-1}{\gamma}} \Big\langle \!\sum_{k=0}^\infty \dfrac{h_k}{\lambda_k - \mu} b_k, \beta \!\Big\rangle \!+\! \dfrac{h_0 \alpha_0}{g-\mu} {\beta}^{-\frac{1}{\gamma}} \eta^{\frac{q+\gamma-1}{\gamma}} \Big\langle \dfrac{\mu_0}{\alpha_0(\lambda_0 - \mu)} b_0 + \sum_{k=1}^\infty \dfrac{\beta_k}{\lambda_k - \mu} b_k, \beta\Big\rangle \\
\notag
&= \dfrac{h_0 \alpha_0(g-\mu+\mu_0)}{(g-\mu)(\lambda_0 - \mu)} {\beta}^{-\frac{1}{\gamma}} \eta^{\frac{q+\gamma-1}{\gamma}}.
\end{align}}
Therefore, recalling that $\mu_0 = \lambda_0 - g$, we obtain
{\allowdisplaybreaks
\begin{align*}
&\mathop{\phantom{=}} (B - \mu)(L-\mu)^{-1}\left(h + \dfrac{h_0 \alpha_0}{g-\mu} {\beta}^{-\frac{1}{\gamma}} \eta^{\frac{q+\gamma-1}{\gamma}}\right)\\& = (L-\mu)(L-\mu)^{-1} \left(h + \dfrac{h_0 \alpha_0}{g-\mu} {\beta}^{-\frac{1}{\gamma}} \eta^{\frac{q+\gamma-1}{\gamma}}\right) - N\Phi (L-\mu)^{-1} \left(h + \dfrac{h_0 \alpha_0}{g-\mu} {\beta}^{-\frac{1}{\gamma}} \eta^{\frac{q+\gamma-1}{\gamma}}\right) = h.
\end{align*}}
This shows \eqref{eq:resolvB}, i.e., it provides an explicit expression of the resolvent of $B$.
Finally, recalling that the operator $(L-\mu)^{-1}$ is well defined for $\mu = g$, since $g$ is not an eigenvalue of $L$ under \eqref{ass:rho} and under the assumption $g>\lambda_1$, we deduce from \eqref{eq:resolvB} that $g$ is a simple pole of the resolvent of $B$.\smallskip

\noindent(iii) Using \eqref{eq:NPhi} and \eqref{eq:w}--\eqref{eq:e00}, it can be immediately proved that $w$ is an eigenfunction of $B$ associated to $g$.\smallskip

\noindent(iv) We use the explicit formula for $P$ given in Remark~\ref{rem:explicitP}. By \eqref{eq:explicitP} and \eqref{eq:resolvB} we have that for any $x \in \calc(\mathbf{S}^1)$, and setting $x_k \coloneqq (x,b_k)_{\dL^2(\mathbf{S}^1)}$:
\begin{equation}\label{eq:explicitP1}
Px = -\dfrac{1}{2\pi\mathrm i} \int_\gamma (B-\mu)^{-1}x \, \d \mu = -\dfrac{1}{2\pi\mathrm i} \int_\gamma (L-\mu)^{-1}\left(x + \dfrac{x_0 \alpha_0}{g-\mu} {\beta}^{-\frac{1}{\gamma}} \eta^{\frac{q+\gamma-1}{\gamma}}\right) \, \d \mu,
\end{equation}
where we take $\gamma$ to be the simple curve given by the positively oriented boundary of the disk centered at $g$, with radius small enough so that it does not encircle any of the eigenvalues $\{\lambda_k\}_{k \in \mathbb{N}}$ of $L$, which are poles of the resolvent $(L-\mu)^{-1}$.
Substituting \eqref{resolvent} into \eqref{eq:explicitP1}, and recalling \eqref{eq:beta0}, \eqref{eq:betak}, we get:
{\allowdisplaybreaks
\begin{align}
\nonumber Px
&= -\dfrac{1}{2\pi\mathrm i} \int_\gamma \sum_{k \in \mathbb N} \frac{x_k}{\lambda_k - \mu} b_k \, \d \mu - \dfrac{x_0\alpha_0}{2\pi\mathrm i} \int_\gamma \sum_{k \in \mathbb N} \frac{(b_k, {\beta}^{-\frac{1}{\gamma}} \eta^{\frac{q+\gamma-1}{\gamma}})_{\dL^2(\mathbf{S}^1)}}{(\lambda_k - \mu)(g-\mu)} b_k \, \d \mu
\label{eq:explicitP2}
\\
&= -\sum_{k \in \mathbb N} x_k b_k \, \dfrac{1}{2\pi\mathrm i} \int_\gamma \frac{1}{\lambda_k - \mu} \, \d \mu -
x_0 \mu_0 b_0 \, \dfrac{1}{2\pi\mathrm i} \int_\gamma \frac{1}{(\lambda_0 - \mu)(g-\mu)} \, \d \mu
\\
\notag
&\qquad - \dfrac{x_0\alpha_0}{2\pi\mathrm i} \sum_{k \in \mathbb N \setminus \{0\}} \beta_k b_k \int_\gamma \frac{1}{(\lambda_k - \mu)(g-\mu)} \, \d \mu.
\end{align}}
Since any $\lambda_k$, {$k \in \mathbb N$}, lies outside of the curve $\gamma$, we have $\int_\gamma \frac{1}{\lambda_k - \mu} \, \d \mu = 0$.
Moreover, since the residues at $g$ of the functions $\frac{1}{(\lambda_k - \mu)(g-\mu)}$, {for $k \in \mathbb N$}, are equal to $-\frac{1}{\lambda_k - g}$, and recalling that $\gamma$ is a simple curve, we get:
\begin{equation*}
\int_\gamma \frac{1}{(\lambda_k - \mu)(g-\mu)} \, \d \mu = -\frac{2\pi\mathrm i}{\lambda_k-g}, \qquad k \in \mathbb N.
\end{equation*}
Therefore, since $\lambda_0 - g = \mu_0$, \eqref{eq:explicitP2} becomes:
\begin{equation*}
Px = \dfrac{x_0 \mu_0}{\lambda_0 - g} b_0 + \sum_{k \in \mathbb N \setminus \{0\}} \dfrac{x_0\alpha_0\beta_k}{\lambda_k - g}  b_k = x_0 \alpha_0 \left(\dfrac{b_0}{\alpha_0} + \sum_{k \in \mathbb N \setminus \{0\}} \dfrac{\beta_k}{\lambda_k - g}  b_k \right) = \langle x, \beta \rangle w.\qedhere
\end{equation*}
\end{proof}
\blu{By the previous results, we are ready to apply Theorem~\ref{th:stability} to the economic problem studied here, obtaining the following.}
\begin{theorem}\label{th:econconv2}
Let \eqref{ass:rho} hold and assume also that $g>\lambda_1 $. Let $K_0\in \calc(\mathbf{S}^1)_{++}$, denote by  $\hat K^{K_0}$  the solution to the linear integro-PDE~\eqref{clebis}. Then,
\begin{equation}\label{eq:convgen}
\left|\hat K^{K_0}_g(t,\cdot) - \langle K_0, \beta \rangle w \right|_\infty \leq M \, \de^{-(g-\lambda_1) t} |K_0 - \langle K_0, \beta \rangle w|_\infty, \ \ \ \forall t\geq 0,
\end{equation}
where $\beta$ is defined by \eqref{eq:betastar0}, $\hat K^{K_0}_g(t,\cdot):= \de^{-gt} \hat K^{K_0}(t,\cdot)$ and
\begin{equation*}
M \coloneqq 1 + |{w}|_\infty \int_0^{2\pi} \beta(\theta) \, \d\theta.
\end{equation*}
In particular,
$$\hat{K}_g^{K_0}(t,\cdot)\stackrel{t\to +\infty}{\longrightarrow} \langle K_0, \beta \rangle w \ \ \ \text{in}\  \ \calc(\mathbf{S}^1).$$
\end{theorem}
\begin{proof}
By Proposition~\ref{prop:spectrumB} and Theorem~\ref{th:stability},  we get
\begin{equation}\label{eq:solconvgen}
\left|\hat K^{K_0}_g(t,\cdot) - \langle K_0, \beta \rangle w \right|_\infty \leq M \, \de^{-\varepsilon t} |K_0 - \langle K_0, \beta \rangle w|_\infty, \ \ \ \forall t\geq 0.
\end{equation}

We now compute explicitly the constants $M$ and $\varepsilon$ appearing in \eqref{eq:solconvgen}. From the proof of Theorem~\ref{th:stability} we know that these two constants come from the estimate provided in \eqref{est1}, which we can now refine.

For the sake of clarity, define $S(t) \coloneqq \de^{tB}, \, t \geq 0$.
Since $w$ is an eigenvector of $B$ with eigenvalue $g$, it follows from \citep[Theorem~3.6, Ch. IV, p. 276]{engelnagel:sgr} that $\de^{gt}$ is an eigenvalue of $S(t)$ with eigenvector $w$.  Using the fact that $\mbox{Range}\, P=\mbox{Span}\{w\}$, we have $S(t)P = \de^{gt} P$, for all $t \geq 0$.
Moreover,  since $\mbox{Range}(1-P)=\mathrm{ker}\, P$, we have $S(t)(1-P) = S(t)_{|\mathrm{ker}\, P} (1-P)$, for all $t \geq 0$,
where $S(t)_{|\mathrm{ker}\, P}$ indicates the restriction of the operator $S(t)$ to $\mathrm{ker}\, P$. 
Hence, denoting by $|\cdot|_{\mathcal{L}(\calc(\mathbf{S}^1))}$ the operator norm on the space of linear continuous operators $\mathcal{L}(\calc(\mathbf{S}^1))$,
\begin{align}\label{eq:stabest}
&\mathop{\phantom{=}}|\de^{-gt}S(t)-P|_{\mathcal{L}(\calc(\mathbf{S}^1))} = |\de^{-gt}(S(t)-S(t)P)|_{\mathcal{L}(\calc(\mathbf{S}^1))} = \de^{-gt}|S(t)-S(t)P|_{\mathcal{L}(\calc(\mathbf{S}^1))}
\\
&=\de^{-gt}|S(t)(1-P)|_{\mathcal{L}(\calc(\mathbf{S}^1))} \leq \de^{-gt}|S(t)_{|\mathrm{ker}\, P}|_{\mathcal{L}(\calc(\mathbf{S}^1))} \, |(1-P)|_{\mathcal{L}(\calc(\mathbf{S}^1))}, \qquad t \geq 0. \notag
\end{align}
We now  estimate $|S(t)_{|\mathrm{ker}\, P}|_{\mathcal{L}(\calc(\mathbf{S}^1))}$ and $|(1-P)|_{\mathcal{L}(\calc(\mathbf{S}^1))}$.

We begin with the first quantity. We aim to prove, with the aid of Hille-Yosida's theorem, that $|S(t)_{|\mathrm{ker}\, P}|$ satisfies
\begin{equation}\label{eq:Stestimate}
|S(t)_{|\mathrm{ker}\, P}|_{\mathcal{L}(\calc(\mathbf{S}^1))} \leq \de^{\lambda_1 t}, \quad t \geq 0.
\end{equation}
From the discussion on subspace semigroups given in \citep[Par. II.2.3]{engelnagel:sgr}, we deduce that the infinitesimal generator of the semigroup $\{S(t)_{|\mathrm{ker}\, P}\}_{t \geq 0}$ is the operator $B_{|\mathrm{ker}\, P} \equiv L_{|\mathrm{ker}\, P}$, whose domain is $D(B_{|\mathrm{ker}\, P}) = D(L_{|\mathrm{ker}\, P}) = D(L) \cap \mathrm{ker}\, P$. 
Spectra of $B_{|\mathrm{ker}\, P}$ and $L_{|\mathrm{ker}\, P}$ must coincide, and clearly $\sigma(L_{|\mathrm{ker}\, P}) \subset \sigma(L)$.
In particular, recalling that $\mathrm{ker}\, P$ is the closed subspace of $\calc(\mathbf{S}^1)$ of functions $x$ such that $\langle x, {\beta} \rangle = 0$, we have that $\sigma(L_{|\mathrm{ker}\, P}) = \{\lambda_k\}_{k \in \mathbb N \setminus\{0\}}$. Indeed, on the one hand, $\lambda_0$ cannot be an eigenvalue of $L_{|\mathrm{ker}\, P}$, since the associated eigenvector $b_0$ does not satisfy $\langle b_0, {\beta}\rangle = 0$; on the other hand, the remaining eigenvalues of $L$ are also eigenvalues of $L_{|\mathrm{ker}\, P}$, since the corresponding eigenvectors belong to $\ker\, P$, and hence are in $D(L_{|\mathrm{ker}\, P}) = D(B_{|\mathrm{ker}\, P})$.
This shows that the resolvent set of $B_{|\mathrm{ker}\, P}$ satisfies $\rho(B_{|\mathrm{ker}\, P}) \supset (\lambda_1, +\infty)$.
Then, in order to apply Hille-Yosida's theorem, it suffices to verify that
\begin{equation}\label{eq:HY}
|(\mu-\lambda_1) (B_{|\mathrm{ker}\, P} - \mu)^{-1}|_{\mathcal{L}(\calc(\mathbf{S}^1))} \leq 1, \quad \text{for } \mu > \lambda_1.
\end{equation}
From \citep[Ch. 7, Sec. 3]{coddington:ode}, we see that either $|(B_{|\mathrm{ker}\, P} - \mu)^{-1}|_{\mathcal{L}(\calc(\mathbf{S}^1))}$ or $-|(B_{|\mathrm{ker}\, P} - \mu)^{-1}|_{\mathcal{L}(\calc(\mathbf{S}^1))}$ must  be an eigenvalue of $(B_{|\mathrm{ker}\, P} - \mu)^{-1}$ and that
\begin{equation*}
\sigma\left((B_{|\mathrm{ker}\, P} - \mu)^{-1}\right) = \left\{\frac{1}{\lambda_k - \mu}\right\}_{k \in \mathbb N \setminus \{0\}}.
\end{equation*}
Therefore, recalling that $\{\lambda_k\}_{k \in \mathbb N \setminus \{0\}}$ is a decreasing sequence and that $\mu > \lambda_1$, we obtain $|(B_{|\mathrm{ker}\, P} - \mu)^{-1}|_{\mathcal{L}(\calc(\mathbf{S}^1))} \leq \frac{1}{\mu - \lambda_1}$, and hence $|(\mu-\lambda_1) (B_{|\mathrm{ker}\, P} - \mu)^{-1}|_{\mathcal{L}(\calc(\mathbf{S}^1))} \leq 1$.

Since we have $\rho(B_{|\mathrm{ker}\, P}) \supset (\lambda_1, +\infty)$ and \eqref{eq:HY}, the assumptions of Hille-Yosida's theorem are verified, and we get \eqref{eq:Stestimate}.

Finally, an estimate for the quantity $|(1-P)|_{\mathcal{L}(\calc(\mathbf{S}^1))}$ can be obtained as follows. Let $x \in \calc(\mathbf{S}^1)$ with $|x|_{\calc(\mathbf{S}^1)} = 1$ and recall that, under the standing assumptions, $\beta \in \calc(\mathbf{S}^1)_{++}$; then
\begin{equation*}
|(1-P)x|_\infty \leq 1 + |\langle x, \beta \rangle| \, |{w}|_\infty \leq 1 + |{w}|_\infty \int_0^{2\pi} \beta(\theta) \, \d\theta.
\end{equation*}
Using the latter together with \eqref{eq:Stestimate} into \eqref{eq:stabest} we obtain the claim.
\end{proof}

\begin{remark}\label{rem:stab}
Theorem~\ref{th:econconv2} deserves some comments. The claim could be proved in the setting $X=\dL^2(\mathbf{S}^1)$ -- the setting of  \cite{BFFGjoeg, BFFGpafa} -- but it would not have useful consequences from the point of view of the application.
\blu{Indeed, the norm of $\dL^2(\mathbf{S}^1)$ cannot control pointwise constraints such as the one prescribed by~\eqref{eq:A++}. Therefore, this result cannot be used to prove that, at least for a suitable set of initial data $K_0$, the optimal control $\hat{C}$ for problem~\eqref{eq:econPb0}, defined in~\eqref{cc}, is also optimal for problem~\eqref{eq:econP}. In the aforementioned references, it is just assumed that $\hat{K}$, solution to \eqref{clebis}, verifies the state constraint~\eqref{eq:A++}. In this way, the authors automatically have that $\hat{C}$ is optimal for problem~\eqref{eq:econP}. Otherwise said, this important issue is skipped from the theoretical point of view and verified only in numerical exercises.} In a different context (delay equations), this issue is approached from the theoretical point of view in \cite{BFG} (although the proof has a gap) and in \cite{BDFG}, where the argument to get the result is very involved.

{
Instead, Theorem~\ref{th:econconv2} has an important consequence in our $X=\mathcal{C}(\mathbf{S}^1)$ setting: indeed, if}
\begin{equation}\label{eq:ststestimategen}
\underline{w} \coloneqq \inf_{\mathbf{S}^1} w>0 \ \ \ \mbox{and} \ \ \  \left(1 + |{w}|_\infty \int_0^{2\pi} \beta(\theta) \, \d\theta\right) |K_0 - \langle K_0, \beta \rangle w|_\infty \leq | \langle K_0, \beta \rangle| \,\,\underline{w},
 \end{equation}
then we deduce that $\hat{K}^{K_0}_g(t,\cdot )>0$, for all  $t\in \R_+$, and the optimal control $\hat{C}$ given in \eqref{cc}
is such that $\hat{C}\in\mathcal{A}_{++}(K_0)$; hence, by Remark~\ref{rem:importante}, it is optimal for problem~\eqref{eq:econP}.
\end{remark}

We conclude this work specializing the result above to the case of homogeneous data, i.e., $A(\cdot) \equiv A > 0$ and $\eta(\cdot)\equiv 1$. This corresponds to the setting presented in \citep{boucekkine:spatialAK}, where, as repeatedly said,  the problem is embedded in $X = \dL^2(\mathbf{S}^1)$. In this case the sequence of eigenvalues and eigenvectors of $L$ is fully explicit. In particular, we have that $\lambda_0 = A$, $\lambda_1= A-\sigma$, $b_0=\frac{1}{\sqrt{2\pi}}\mathbf{1}_{\mathbf{S}^1}$. 
Rewriting the statement of Theorem~\ref{th:econconv2} using $\lambda_0$, $\lambda_1$, and $b_0$ above, we get the following corollary.
\begin{corollary}\label{th:econconvconst}
Let $K_0\in X_{++}$ and let $\hat K^{K_0}$ be the solution to the linear integro-PDE~\eqref{clebis}. Assume that $A(1-\gamma) < \rho < A(1-\gamma) + \sigma \gamma$.
Then,
\begin{equation}\label{eq:conv}
\Big|\hat K^{K_0}_g(t) -\frac{1}{2\pi} \int_{\mathbf{S}^1} K_0(\theta) \, \d \theta\,\Big|_{\calc(\mathbf{S}^1)} \leq 2 \, \de^{-(g-A+\sigma) t}\,\, \Bigl|K_0 - \frac{1}{2\pi} \int_{\mathbf{S}^1} K_0(\theta) \, \d \theta \,\Bigr|_{\calc(\mathbf{S}^1)}, \quad \forall t\geq 0,
\end{equation}
where $\hat K^{K_0}_g(t):= \de^{-gt} \hat K^{K_0}(t)$.
In particular,
$$\hat{K}_g^{K_0}(t)\stackrel{t\to +\infty}{\longrightarrow} \frac{1}{2\pi} \int_{\mathbf{S}^1} K_0(\theta) \, \d \theta\, \quad \text{in}\ \mathcal{C}(\mathbf{S}^1),$$ and, if
\begin{equation}\label{eq:ststestimate}
 2 \,\,\Bigl|K_0 - \frac{1}{2\pi} \int_{\mathbf{S}^1} K_0(\theta) \, \d \theta \Bigr|_{\calc(\mathbf{S}^1)} \leq\frac{1}{2\pi} \int_{\mathbf{S}^1}K_0(\theta) \, \d \theta,
\end{equation}
we have $\hat{K}_g(t)\in X_{++}$ for all  $t\in \R_+$. Hence, the optimal control for problem $(P^{\frac{1}{\sqrt{2\pi}}\mathbf{1}_{\mathbf{S}^1}})$, i.e.,
\begin{equation}\label{ccconst}
\hat{C}(t,\theta)=\de^{gt} \, \frac{A-g}{2\pi}\int_{\mathbf{S}^1} {K}_0(\xi) \, \d\xi,
\end{equation}
belongs to $\mathcal{A}_{++}(K_0)$ and therefore, by Remark~\ref{rem:importante}, it is optimal for problem $(P)$.
\end{corollary}
\blu{
\begin{proof}
We need just to rewrite \eqref{eq:convgen} in the present setting to get \eqref{eq:conv}. The remaining statements follow immediately (see also Remark~\ref{rem:stab}).
Notice that, in the present setting, the coefficients $\beta_{k}$ defined in \eqref{eq:betak} vanish for all $k\geq 1$; hence,
$$
w=\frac{b_{0}}{\alpha_{0}}.
$$
Therefore, since $b_0=\frac{1}{\sqrt{2\pi}}\mathbf{1}_{\mathbf{S}^1}$, we get
\begin{equation*}
\langle K_0, \beta \rangle w = \alpha_0 \langle K_0, b_0 \rangle  \frac{b_{0}}{\alpha_{0}} = \frac{1}{2\pi} \int_{\mathbf{S}^1} K_0(\theta) \, \d \theta.
\end{equation*}
Moreover, the constant $M$ appearing in \eqref{eq:convgen}, satisfies
\begin{equation*}
M = 1 + |{w}|_\infty \int_0^{2\pi} \beta(\theta) \, \d\theta = 1 + \left|\frac{b_0}{\alpha_0}\right|_\infty \int_0^{2\pi} \alpha_0 b_0(\theta) \, \d\theta = 2.
\end{equation*}
Finally, recalling that $\lambda_1 = A - \sigma$, rewriting \eqref{eq:convgen} we get \eqref{eq:conv}.
\end{proof}
}
\blu{\section{Conclusions and potential extensions}\label{sec:conclusion}
In this paper we studied a class of optimal control problems in infinite dimension with a positivity state constraint, motivated by economic applications. With respect to the past stream of literature, dealing with the \mbox{$\dL^{2}$ setting}, we provided a more general abstract framework, that includes the space of continuous functions. This allows to treat rigorously the positivity state constraint, by introducing a suitable auxiliary optimal control problem and verifying the admissibility of its optimal path for the original problem. 
As explained in the Introduction, this was an issue that in previous works was left out at the theoretical level and checked only numerically (see Remark~\ref{rem:stab}). In Section~\ref{sec:app} we showed, instead, that by formulating the problem in the space of continuous functions it is possible to deal successfully with the positivity state constraint, thanks to the stability results given in Section~\ref{sec:ss}.
}

\blu{
Our approach is based on the particular structure of the problem, that allows to find explicit solutions to the HJB equation of the auxiliary problem, as shown in Section~\ref{sec:hjbexpl}. Indeed, the state equation~\eqref{state} is linear, in~\eqref{eq:Nexpl} we choose the bounded and positive linear operator $N$ appearing in~\eqref{state} to be of multiplicative type, and the function $u$ in~\eqref{eq:uexpl} is a power utility function.
}

\blu{Nonetheless, it may be possible to adapt our approach to treat optimal control problems with a positivity state constraint, where nonlinear state equations and/or non-homogeneous utility functions appear, as in various economic models. To this end, we suggest a possible scheme for future work in the case where a semilinear state equation is considered, i.e., where an additional term $F(x(t))$, depending only on the state variable $x$, appears in the state equation~\eqref{state}. 
\begin{enumerate}[(i)]
  \item A nonlinear semigroup can be associated to semilinear abstract equations and it is necessary to impose that it preserves strict positivity, similarly to the requirement of point~\textit{(\ref{hyp:L})} of Assumption~\ref{ass:L}.
  \item The pointwise positivity state constraint still makes the HJB equation associated to the optimal control problem very difficult to study. Therefore, we can consider an auxiliary problem in a half-space containing the positive cone of the state space $X$.
  \item It is now necessary to prove existence, uniqueness, and regularity of the solutions to the HJB equation of the auxiliary optimal control problem: this may be very hard, but there are results of this type in some cases --- see, e.g., \cite{FT, FGG1,FGG2}, where the concept of viscosity solutions is used.
  \item After the previous point is settled, a verification theorem should be established, as we did in Theorem~\ref{verification}: its proof does not require to find explicit solutions to the HJB equation of the auxiliary optimal control problem --- see again \cite{FT, FGG1,FGG2}.
  \item Finally, one should prove that under suitable assumptions the solution to the auxiliary optimal control problem remains in the positive cone of the state space $X$, as happens in our setting if the \emph{a priori} estimate~\eqref{eq:ststestimategen} holds, by virtue of Theorem~\eqref{th:econconv2}. This could be achieved by extending the stability results given in Section~\ref{sec:ss} to the case of nonlinear semigroups. 
  \end{enumerate}}

\bibliographystyle{plainnat}
\bibliography{Bibliography}

\begin{thebibliography}{39}
\providecommand{\natexlab}[1]{#1}
\providecommand{\url}[1]{\texttt{#1}}
\expandafter\ifx\csname urlstyle\endcsname\relax
  \providecommand{\doi}[1]{doi: #1}\else
  \providecommand{\doi}{doi: \begingroup \urlstyle{rm}\Url}\fi

\bibitem[Addona et~al.(2020)Addona, Bandini, and Masiero]{addona2020:BE}
D.~Addona, E.~Bandini, and F.~Masiero.
\newblock A nonlinear {B}ismut-{E}lworthy formula for {HJB} equations with
  quadratic {H}amiltonian in {B}anach spaces.
\newblock \emph{NoDEA Nonlinear Differential Equations Appl.}, 27\penalty0
  (4):\penalty0 Paper No. 37, 56, 2020.
\newblock \doi{10.1007/s00030-020-00639-7}.

\bibitem[Arendt et~al.(1986)Arendt, Grabosch, Greiner, Groh, Lotz, Moustakas,
  Nagel, Neubrander, and Schlotterbeck]{arendt1986:possgr}
W.~Arendt, A.~Grabosch, G.~Greiner, U.~Groh, H.~P. Lotz, U.~Moustakas,
  R.~Nagel, F.~Neubrander, and U.~Schlotterbeck.
\newblock \emph{One-parameter semigroups of positive operators}, volume 1184 of
  \emph{Lecture Notes in Mathematics}.
\newblock Springer-Verlag, Berlin, 1986.

\bibitem[Bambi et~al.(2012)Bambi, Fabbri, and Gozzi]{BFG}
M.~Bambi, G.~Fabbri, and F.~Gozzi.
\newblock Optimal policy and consumption smoothing effects in the time-to-build
  {AK} model.
\newblock \emph{Econom. Theory}, 50\penalty0 (3):\penalty0 635--669, 2012.
\newblock \doi{10.1007/s00199-010-0577-3}.

\bibitem[Bambi et~al.(2017)Bambi, Di~Girolami, Federico, and Gozzi]{BDFG}
M.~Bambi, C.~Di~Girolami, S.~Federico, and F.~Gozzi.
\newblock Generically distributed investments on flexible projects and
  endogenous growth.
\newblock \emph{Econom. Theory}, 63\penalty0 (2):\penalty0 521--558, 2017.
\newblock \doi{10.1007/s00199-015-0946-z}.

\bibitem[Bensoussan et~al.(2007)Bensoussan, Da~Prato, Delfour, and
  Mitter]{BDDM}
A.~Bensoussan, G.~Da~Prato, M.~C. Delfour, and S.~K. Mitter.
\newblock \emph{Representation and control of infinite dimensional systems}.
\newblock Systems \& Control: Foundations \& Applications. Birkh\"{a}user
  Boston, Inc., Boston, MA, second edition, 2007.

\bibitem[Boucekkine et~al.(2005)Boucekkine, Licandro, Puch, and del
  Rio]{BoucekkineetalJET2005}
R.~Boucekkine, O.~Licandro, L.~A. Puch, and F.~del Rio.
\newblock Vintage capital and the dynamics of the {AK} model.
\newblock \emph{J. Econom. Theory}, 120\penalty0 (1):\penalty0 39--72, 2005.
\newblock \doi{10.1016/j.jet.2004.02.006}.

\bibitem[Boucekkine et~al.(2013)Boucekkine, Camacho, and
  Fabbri]{boucekkine:spatialAK}
R.~Boucekkine, C.~Camacho, and G.~Fabbri.
\newblock Spatial dynamics and convergence: the spatial {AK} model.
\newblock \emph{J. Econom. Theory}, 148\penalty0 (6):\penalty0 2719--2736,
  2013.
\newblock \doi{10.1016/j.jet.2013.09.013}.

\bibitem[Boucekkine et~al.(2018)Boucekkine, Fabbri, Federico, and
  Gozzi]{BFFGjoeg}
R.~Boucekkine, G.~Fabbri, S.~Federico, and F.~Gozzi.
\newblock Growth and agglomeration in the heterogeneous space: a generalized
  {AK} approach.
\newblock \emph{Journal of Economic Geography}, 19\penalty0 (6):\penalty0
  1287--1318, 2018.
\newblock \doi{10.1093/jeg/lby041}.

\bibitem[Boucekkine et~al.(To appear)Boucekkine, Fabbri, Federico, and
  Gozzi]{BFFGpafa}
R.~Boucekkine, G.~Fabbri, S.~Federico, and F.~Gozzi.
\newblock Control theory in infinite dimension for the optimal location of
  economic activity: The role of social welfare function.
\newblock \emph{Pure and Applied Functional Analysis}, To appear.

\bibitem[Brown et~al.(2013)Brown, Eastham, and Schmidt]{brown:periodic}
B.~M. Brown, M.~S.~P. Eastham, and K.~M. Schmidt.
\newblock \emph{Periodic differential operators}, volume 230 of \emph{Operator
  Theory: Advances and Applications}.
\newblock Birkh\"{a}user/Springer Basel AG, Basel, 2013.
\newblock \doi{10.1007/978-3-0348-0528-5}.

\bibitem[Calvia(2018)]{calvia:MCcontrol}
A.~Calvia.
\newblock Optimal control of continuous-time {M}arkov chains with noise-free
  observation.
\newblock \emph{SIAM J. Control Optim.}, 56\penalty0 (3):\penalty0 2000--2035,
  2018.
\newblock \doi{10.1137/17M1139989}.

\bibitem[Calvia(2020)]{calvia:filtcontrol}
A.~Calvia.
\newblock Stochastic filtering and optimal control of pure jump {M}arkov
  processes with noise-free partial observation.
\newblock \emph{ESAIM Control Optim. Calc. Var.}, 26:\penalty0 Paper No. 25,
  47, 2020.
\newblock \doi{10.1051/cocv/2019020}.

\bibitem[Cannarsa and Di~Blasio(1995)]{CannarsaDiBlasio95}
P.~Cannarsa and G.~Di~Blasio.
\newblock A direct approach to infinite-dimensional {H}amilton-{J}acobi
  equations and applications to convex control with state constraints.
\newblock \emph{Differential Integral Equations}, 8\penalty0 (2):\penalty0
  225--246, 1995.

\bibitem[Cannarsa et~al.(1991)Cannarsa, Gozzi, and Soner]{CannarsaGozziSoner91}
P.~Cannarsa, F.~Gozzi, and H.~M. Soner.
\newblock A boundary value problem for {H}amilton-{J}acobi equations in
  {H}ilbert spaces.
\newblock \emph{Appl. Math. Optim.}, 24\penalty0 (2):\penalty0 197--220, 1991.
\newblock \doi{10.1007/BF01447742}.

\bibitem[Capuzzo-Dolcetta and Lions(1990)]{capuzzolions1990:hjb}
I.~Capuzzo-Dolcetta and P.-L. Lions.
\newblock Hamilton-{J}acobi equations with state constraints.
\newblock \emph{Trans. Amer. Math. Soc.}, 318\penalty0 (2):\penalty0 643--683,
  1990.
\newblock \doi{10.2307/2001324}.

\bibitem[Cl\'{e}ment et~al.(1987)Cl\'{e}ment, Heijmans, Angenent, van Duijn,
  and de~Pagter]{clement1987:sgr}
Ph. Cl\'{e}ment, H.~J. A.~M. Heijmans, S.~Angenent, C.~J. van Duijn, and
  B.~de~Pagter.
\newblock \emph{One-parameter semigroups}, volume~5 of \emph{CWI Monographs}.
\newblock North-Holland Publishing Co., Amsterdam, 1987.

\bibitem[Coddington and Levinson(1955)]{coddington:ode}
E.~A. Coddington and N.~Levinson.
\newblock \emph{Theory of ordinary differential equations}.
\newblock McGraw-Hill Book Company, Inc., New York-Toronto-London, 1955.

\bibitem[Crandall and Lions(1985)]{crandall85:HJB1}
M.~G. Crandall and P.-L. Lions.
\newblock Hamilton-{J}acobi equations in infinite dimensions. {I}. {U}niqueness
  of viscosity solutions.
\newblock \emph{J. Funct. Anal.}, 62\penalty0 (3):\penalty0 379--396, 1985.
\newblock \doi{10.1016/0022-1236(85)90011-4}.

\bibitem[Crandall and Lions(1986{\natexlab{a}})]{crandall86:HJB2}
M.~G. Crandall and P.-L. Lions.
\newblock Hamilton-{J}acobi equations in infinite dimensions. {II}. {E}xistence
  of viscosity solutions.
\newblock \emph{J. Funct. Anal.}, 65\penalty0 (3):\penalty0 368--405,
  1986{\natexlab{a}}.
\newblock \doi{10.1016/0022-1236(86)90026-1}.

\bibitem[Crandall and Lions(1986{\natexlab{b}})]{crandall86:HJB3}
M.~G. Crandall and P.-L. Lions.
\newblock Hamilton-{J}acobi equations in infinite dimensions. {III}.
\newblock \emph{J. Funct. Anal.}, 68\penalty0 (2):\penalty0 214--247,
  1986{\natexlab{b}}.
\newblock \doi{10.1016/0022-1236(86)90005-4}.

\bibitem[Engel and Nagel(2000)]{engelnagel:sgr}
K.-J. Engel and R.~Nagel.
\newblock \emph{One-parameter semigroups for linear evolution equations},
  volume 194 of \emph{Graduate Texts in Mathematics}.
\newblock Springer-Verlag, New York, 2000.
\newblock With contributions by S. Brendle, M. Campiti, T. Hahn, G. Metafune,
  G. Nickel, D. Pallara, C. Perazzoli, A. Rhandi, S. Romanelli and R.
  Schnaubelt.

\bibitem[Fabbri and Gozzi(2008)]{FabbriGozzi08}
G.~Fabbri and F.~Gozzi.
\newblock Solving optimal growth models with vintage capital: the dynamic
  programming approach.
\newblock \emph{J. Econom. Theory}, 143\penalty0 (1):\penalty0 331--373, 2008.
\newblock \doi{10.1016/j.jet.2008.03.008}.

\bibitem[Fabbri et~al.(2017)Fabbri, Gozzi, and Swiech]{fabbri:soc}
G.~Fabbri, F.~Gozzi, and A.~Swiech.
\newblock \emph{Stochastic optimal control in infinite dimension}, volume~82 of
  \emph{Probability Theory and Stochastic Modelling}.
\newblock Springer, Cham, 2017.
\newblock Dynamic programming and HJB equations, With a contribution by Marco
  Fuhrman and Gianmario Tessitore.

\bibitem[Faggian(2008)]{Faggian08}
S.~Faggian.
\newblock Hamilton-{J}acobi equations arising from boundary control problems
  with state constraints.
\newblock \emph{SIAM J. Control Optim.}, 47\penalty0 (4):\penalty0 2157--2178,
  2008.
\newblock \doi{10.1137/070683738}.

\bibitem[Fattorini(1999)]{Fattorini99}
H.~O. Fattorini.
\newblock \emph{Infinite-dimensional optimization and control theory},
  volume~62 of \emph{Encyclopedia of Mathematics and its Applications}.
\newblock Cambridge University Press, Cambridge, 1999.
\newblock \doi{10.1017/CBO9780511574795}.

\bibitem[Federico and Tacconi(2014)]{FT}
S.~Federico and E.~Tacconi.
\newblock Dynamic programming for optimal control problems with delays in the
  control variable.
\newblock \emph{SIAM J. Control Optim.}, 52\penalty0 (2):\penalty0 1203--1236,
  2014.
\newblock \doi{10.1137/110840649}.

\bibitem[Federico et~al.(2010)Federico, Goldys, and Gozzi]{FGG1}
S.~Federico, B.~Goldys, and F.~Gozzi.
\newblock H{JB} equations for the optimal control of differential equations
  with delays and state constraints, {I}: regularity of viscosity solutions.
\newblock \emph{SIAM J. Control Optim.}, 48\penalty0 (8):\penalty0 4910--4937,
  2010.
\newblock \doi{10.1137/09076742X}.

\bibitem[Federico et~al.(2011)Federico, Goldys, and Gozzi]{FGG2}
S.~Federico, B.~Goldys, and F.~Gozzi.
\newblock H{JB} equations for the optimal control of differential equations
  with delays and state constraints, {II}: {V}erification and optimal
  feedbacks.
\newblock \emph{SIAM J. Control Optim.}, 49\penalty0 (6):\penalty0 2378--2414,
  2011.
\newblock \doi{10.1137/100804292}.

\bibitem[Feichtinger et~al.(2006)Feichtinger, Hartl, Kort, and
  Veliov]{feichtinger2006:vintage}
G.~Feichtinger, R.~F. Hartl, P.~M. Kort, and V.~M. Veliov.
\newblock Anticipation effects of technological progress on capital
  accumulation: a vintage capital approach.
\newblock \emph{Journal of {E}conomic {T}heory}, 126\penalty0 (1):\penalty0
  143--164, 2006.

\bibitem[Freni et~al.(2006)Freni, Gozzi, and Salvadori]{FreniGozziSalvadori06}
G.~Freni, F.~Gozzi, and N.~Salvadori.
\newblock Existence of optimal strategies in linear multisector models.
\newblock \emph{Econom. Theory}, 29\penalty0 (1):\penalty0 25--48, 2006.

\bibitem[Fuhrman et~al.(2010)Fuhrman, Masiero, and Tessitore]{fuhrman2010:SDDE}
M.~Fuhrman, F.~Masiero, and G.~Tessitore.
\newblock Stochastic equations with delay: optimal control via {BSDE}s and
  regular solutions of {H}amilton-{J}acobi-{B}ellman equations.
\newblock \emph{SIAM J. Control Optim.}, 48\penalty0 (7):\penalty0 4624--4651,
  2010.
\newblock \doi{10.1137/080730354}.

\bibitem[Hartl et~al.(1995)Hartl, Sethi, and Vickson]{hartl1995:mp}
R.~F. Hartl, S.~P. Sethi, and R.~G. Vickson.
\newblock A survey of the maximum principles for optimal control problems with
  state constraints.
\newblock \emph{SIAM Rev.}, 37\penalty0 (2):\penalty0 181--218, 1995.
\newblock \doi{10.1137/1037043}.

\bibitem[Katsoulakis(1994)]{Katsoulakis94}
M.~A. Katsoulakis.
\newblock Viscosity solutions of second order fully nonlinear elliptic
  equations with state constraints.
\newblock \emph{Indiana Univ. Math. J.}, 43\penalty0 (2):\penalty0 493--519,
  1994.
\newblock \doi{10.1512/iumj.1994.43.43020}.

\bibitem[Kocan and Soravia(1998)]{KocanSoravia98}
M.~Kocan and P.~Soravia.
\newblock A viscosity approach to infinite-dimensional {H}amilton-{J}acobi
  equations arising in optimal control with state constraints.
\newblock \emph{SIAM J. Control Optim.}, 36\penalty0 (4):\penalty0 1348--1375,
  1998.
\newblock \doi{10.1137/S0363012996301622}.

\bibitem[Li and Yong(1995)]{LY}
X.~J. Li and J.~M. Yong.
\newblock \emph{Optimal control theory for infinite-dimensional systems}.
\newblock Systems \& Control: Foundations \& Applications. Birkh\"{a}user
  Boston, Inc., Boston, MA, 1995.
\newblock \doi{10.1007/978-1-4612-4260-4}.

\bibitem[Masiero(2008)]{masiero2008:soc}
F.~Masiero.
\newblock Stochastic optimal control problems and parabolic equations in
  {B}anach spaces.
\newblock \emph{SIAM J. Control Optim.}, 47\penalty0 (1):\penalty0 251--300,
  2008.
\newblock \doi{10.1137/050632725}.

\bibitem[Masiero(2016)]{masiero2016:hjb}
F.~Masiero.
\newblock H{JB} equations in infinite dimensions under weak regularizing
  properties.
\newblock \emph{J. Evol. Equ.}, 16\penalty0 (4):\penalty0 789--824, 2016.
\newblock \doi{10.1007/s00028-015-0320-4}.

\bibitem[Soner(1986)]{soner:optcontrol1}
H.~M. Soner.
\newblock Optimal control with state-space constraint. {I}.
\newblock \emph{SIAM J. Control Optim.}, 24\penalty0 (3):\penalty0 552--561,
  1986.

\bibitem[Soner(1988)]{soner88:HJB}
H.~M. Soner.
\newblock On the {H}amilton-{J}acobi-{B}ellman equations in {B}anach spaces.
\newblock \emph{J. Optim. Theory Appl.}, 57\penalty0 (3):\penalty0 429--437,
  1988.
\newblock \doi{10.1007/BF02346162}.

\end{thebibliography}
\end{document}